\documentclass[11pt,a4paper]{article}
\usepackage{latexsym,amsfonts,amsmath,amsthm,graphics}
\usepackage[normalem]{ulem}
\usepackage[shortlabels]{enumitem}
\usepackage[dvipsnames]{xcolor}
\usepackage{float}
\usepackage{pgfplots}
\usepackage{subcaption}
\usepackage{tikz}
\usepackage{pdfsync}

\usepackage{booktabs,caption}
\usepackage{multirow}

\pgfplotsset{every tick label/.append style={font=\scriptsize}}
\newenvironment{customlegend}[1][]{
	\begingroup
	\csname pgfplots@init@cleared@structures\endcsname
	\pgfplotsset{#1}
}{
\csname pgfplots@createlegend\endcsname
\endgroup
}

\def\addlegendimage{\csname pgfplots@addlegendimage\endcsname}
\pgfplotsset{
	cycle list={%
		{draw=black,mark=star,solid},
		{draw=black, mark=square,solid},
		{draw=black,mark=+,solid},
		{black,mark=o},
		{draw=black, mark=none,solid}}
}

\definecolor{mycolor1}{rgb}{0.00000,0.44700,0.74100}
\definecolor{mycolor2}{rgb}{0.85000,0.32500,0.09800}
\definecolor{mycolor3}{rgb}{0.49400,0.18400,0.55600}
\definecolor{mycolor4}{rgb}{0.12157,0.49804,0.31373}
\definecolor{mycolor5}{rgb}{0.85000,0.11800,0.09800}

\newtheorem{theorem}{Theorem}

\newtheorem{corollary}[theorem]{Corollary}

\newtheorem{algorithm}{Algorithm}
\newtheorem{remark}[theorem]{Remark}

\newcommand{\norm}[1]{\left\Vert#1\right\Vert}
\newcommand{\To}{\rightarrow}

\newcommand{\bsDelta}{\boldsymbol{\Delta}}

\newcommand{\bsgamma}{{\boldsymbol{\gamma}}}

\newcommand{\bsxi}{\boldsymbol{\xi}}

\newcommand{\bsv}{\boldsymbol{v}}

\newcommand{\bsx}{\boldsymbol{x}}
\newcommand{\bsh}{\boldsymbol{h}}
\newcommand{\bsg}{\boldsymbol{g}}
\newcommand{\bsq}{\boldsymbol{q}}
\newcommand{\bst}{\boldsymbol{t}}
\newcommand{\bsz}{\boldsymbol{z}}
\newcommand{\bsu}{\boldsymbol{u}}
\newcommand{\bsone}{\boldsymbol{1}}

\newcommand{\bsy}{\boldsymbol{y}}

\newcommand{\D}{{\cal D}}

\newcommand{\cP}{{\cal P}}

\newcommand{\wal}{{\rm wal}}
\newcommand{\sob}{{\rm sob}}

\newcommand\setu{{\mathfrak{u}}}
\newcommand{\uu}{\mathfrak{u}}

\newcommand{\icomp}{\mathtt{i}}
\newcommand{\bszero}{\boldsymbol{0}}
\newcommand{\rd}{\,\mathrm{d}}
\newcommand{\tr}{\mathrm{tr}}
\newcommand{\NN}{\mathbb{N}}
\newcommand{\ZZ}{\mathbb{Z}}

\newcommand{\RR}{\mathbb{R}}
\newcommand{\CC}{\mathbb{C}}
\newcommand{\FF}{\mathbb{F}}

\newcommand{\calH}{\mathcal{H}}

\renewcommand{\pmod}[1]{\,(\bmod\,#1)}

\newcommand{\abs}[1]{\left\vert#1\right\vert}

\newcommand{\calP}{\mathcal{P}}

\newcommand{\calZ}{{\mathcal{Z}}}

\newcommand{\cH}{{\cal H}}

\newcommand{\N}{{\mathbb{N}}} 
\newcommand{\Z}{{\mathbb{Z}}} 

\newcommand{\setv}{{\mathfrak{v}}}
\newcommand*{\bfrac}[2]{\genfrac{}{}{0pt}{}{#1}{#2}}

\allowdisplaybreaks

\begin{document}

\title{Constructing lattice points for numerical integration by a reduced fast successive coordinate search algorithm}

\author{Adrian Ebert, Peter Kritzer}

\date{\today}

\maketitle

\begin{abstract}
\noindent In this paper, we study an efficient algorithm for constructing node sets of high-quality quasi-Monte Carlo integration rules 
for weighted Korobov, Walsh, and Sobolev spaces. The algorithm presented is a reduced fast successive coordinate search (SCS) algorithm, which is adapted to situations where the weights in the function space show a sufficiently fast decay. The new SCS algorithm is designed to work for the construction of lattice points, and, in a modified version, for polynomial lattice points, and the corresponding integration rules can be used to treat functions in different kinds of function spaces. We show that the integration rules constructed by our algorithms satisfy error bounds of optimal convergence order. Furthermore, we give details on efficient implementation such that we obtain a considerable speed-up of previously known SCS algorithms. This improvement is illustrated by numerical results. The speed-up obtained by our results may be of particular interest in the context of QMC for PDEs with random coefficients, where both the dimension and the required number of points are usually very large. Furthermore, our main theorems yield previously unknown generalizations of earlier results.
\end{abstract}



\section{Introduction}  \label{secIntro}

Quasi-Monte Carlo (QMC) rules are equal-weight integration rules that are used for approximating integrals of functions 
over $[0,1]^s$, 
\[
\frac{1}{N} \sum_{n=0}^{N-1} f(\bsx_n)\approx \int_{[0,1]^s } f(\bsx) \rd \bsx.
\]
As opposed to Monte Carlo rules, where the integration nodes $\bsx_0,\ldots,\bsx_{N-1}$ are selected at random, QMC integration 
is based on the idea of deterministically choosing the integration node set $\calP=\{\bsx_0,\ldots,\bsx_{N-1}\}$; here, the set $\calP$ 
is interpreted as a multi-set, i.e., points are considered taking their multiplicity into account. For introductions to QMC methods 
and their applications we refer to \cite{DKS13, DP10, L09, LP14, N92b}. 

Modern approaches to efficient QMC methods usually consider numerical integration for elements of Banach spaces, or, using a narrower setting, as 
in the present paper, for elements of certain reproducing kernel Hilbert spaces $(\calH, \norm{\cdot}_{\calH})$. For further information on reproducing 
kernel Hilbert spaces, 
see \cite{A50}, and for details on the relation between such spaces and QMC theory, we refer to \cite{SW98, SW01}. In this context, the criterion 
considered for assessing 
the quality of a QMC integration rule based on a node set $\calP = \{\bsx_0,\ldots,\bsx_{N-1}\}$ for integration in a space 
$(\calH, \norm{\cdot}_{\calH})$ 
is the worst-case error,
\[
e_{N,s} (\calH,\calP):=\sup_{\substack{f\in\calH \\ \norm{f}_{\calH}\le 1}} \abs{\int_{[0,1]^s} f(\bsx)\rd\bsx - \frac{1}{N}\sum_{n=0}^{N-1} f(\bsx_n)}.
\]

In this paper, we investigate special types of QMC rules, namely lattice rules (see, e.g., \cite{N92b, SJ94} for introductions) and polynomial lattice rules. Here, we solely consider rank-1 lattice rules which are based 
on the choice of a positive integer $N$ and a so-called generating vector $\bsz\in \{0,1,\ldots,N-1\}^s$. 
Using these parameters, an $N$-element lattice point set is given by the points 
\[
\bsx_n:=\left\{\frac{n\bsz}{N}\right\},\quad 0\le n\le N-1.
\]
Here, we write $\{x\}=x-\lfloor x \rfloor$ for real numbers $x$, and apply $\{\cdot\}$ componentwise for vectors. 
Further details on these point sets and the function spaces whose elements 
can be integrated numerically using lattice rules will be given below in Section \ref{secKor}.

Polynomial lattice rules, see, e.g., \cite{DP10, N92b} are of a similar structure as lattice rules, 
but arithmetic over the reals is replaced by arithmetic of polynomials over finite fields. 
We will give further details on polynomial lattice rules in Section \ref{secWal}.

Returning to lattice rules, the crucial question regarding these integration rules is how to find a generating vector $\bsz$ 
that guarantees a low worst-case error of integration in a given function space. In general, there are no explicit constructions of good generating vectors for dimensions $s \ge 3$. 
One way to find good generating vectors is the component-by-component (CBC) construction, which is based on greedy algorithms choosing one component of the generating vector at a time. 
It was shown in \cite{K03} for prime $N$ and in \cite{D04} for non-prime $N$ that it is possible to find generating vectors yielding essentially optimal results for certain spaces of $s$-variate 
functions by the CBC construction. Furthermore, it was shown in \cite{NC06b,NC06} that the computational cost of these algorithms is of order $\mathcal{O} (s N\log N)$. 
While the technique outlined in \cite{NC06b,NC06} is very sophisticated, and the computational cost of order $\mathcal{O} (s N\log N)$ 
is excellent in comparison to previously known results, there is one drawback that remains. For $s$ and $N$ that are simultaneously large this cost may be still too high to construct $\bsz$. 
This is for example the case in recently analyzed PDE applications, see, e.g. \cite{DKLGNS14} and \cite{KSS12}, in which the quantity of interest is given as an infinite-dimensional integral which 
is approximated by a very high-dimensional integral using a large number of (polynomial) lattice points. In the paper \cite{DKLP15} it was therefore shown that this order of magnitude can be reduced further under 
suitable circumstances. The idea underlying the main result in \cite{DKLP15} is to use the concept of weighted function spaces in the CBC construction. We will now shortly comment on weighted spaces and tractability, 
in order to describe the general idea of the paper \cite{DKLP15} and also of the present paper.

The idea to use weighted function spaces in the context of quasi-Monte Carlo methods was introduced in the seminal paper \cite{SW98}. Motivated by applications 
from financial mathematics, where different variables may have very different influence on a computational problem, Sloan and Wo\'{z}niakowski introduced additional parameters in the definition 
of the function spaces under consideration, namely weights. These are given by a set of nonnegative real numbers 
$(\gamma_{\setu})_{\setu\subseteq [s]}$. Here and in the following, we write $[s]$ 
to denote the index set $\{1,\ldots,s\}$. The weight $\gamma_{\setu}$ models the importance of the projection of a given integrand $f$ in the function space onto the variables $x_j$ with $j\in\setu$. 
A small value of $\gamma_{\setu}$ means that the corresponding group of variables has only little influence on the problem, whereas a large value of $\gamma_{\setu}$ means the opposite. A special 
but important subcase is the case of product weights, where $\gamma_{\setu}=\prod_{j\in\setu} \gamma_j$ for a (usually non-increasing) sequence $(\gamma_j)_{j\ge 1}$ of positive integers. In this case, $\gamma_j$ can be thought of as modeling the influence of the variable $x_j$. 

The effect of studying weighted spaces in integration problems is that, if the influence of the variables (or, in other words, the weights) 
in the problem decay sufficiently fast for coordinates with high indices, one can vanquish the curse of dimensionality that is inherent to many 
high-dimensional problems. Indeed, under certain summability conditions on the weights, it is even possible to obtain bounds on the integration error that do not depend on the dimension of 
the problem at all. This is a property known as tractability, and we refer to the trilogy of Novak and Wo\'{z}niakowski \cite{NW08}--\cite{NW12} for extensive information on this subject.

The paper \cite{DKLP15} incorporated the weights of a given function space in the CBC construction of lattice rules that yield a low integration error for the same function space. Indeed, 
depending on the weights, the size of the search space for each component of the generating vector $\bsz$ was adjusted to the corresponding coordinate weight. This reduction is the motivation 
for calling the modified CBC algorithm from \cite{DKLP15} a ``reduced'' CBC construction. It was also shown in \cite{DKLP15} that the reduced CBC construction can be adapted to the existing 
fast CBC construction of Nuyens and Cools. Furthermore, it was shown that in the case of sufficiently fast decaying product weights the computational cost of the 
resulting reduced fast CBC construction can be independent of the dimension. These results also hold analogously for the case of polynomial lattice rules.

\bigskip

A different modification of the fast CBC construction was presented in the recent paper \cite{ELN18}, where a so-called successive coordinate search (SCS) algorithm was presented. In this 
approach, one starts with a given generating vector $\bsz^0$ of a lattice rule. Then, the single components of this starting vector are 
improved on a step-by-step basis. The difference in 
the SCS approach, as opposed to the CBC approach, is that the algorithm has the starting vector as an input and the generating vector is not constructed from scratch. In particular, one could 
use the output of the fast CBC algorithm (or alternatively, of a previous instance of the SCS algorithm) as the input for the SCS algorithm and thereby further improve on the quality of the corresponding 
lattice rule. It is also possible to have a fast implementation of the SCS algorithm which has a computational cost of $\mathcal{O}(s N\log N)$, which is the same as that of the fast CBC construction. 
The paper \cite{ELN18} contains, apart from a theoretical analysis of the algorithm, also numerical results on the performance of the SCS algorithms. The numerical results show that the SCS algorithm 
can yield a significant improvement of the CBC algorithm for particular parameter settings (in particular, the performance is influenced by the choice of weights $\gamma_{\setu}$ in the problem).

\bigskip

In the present paper, we would like to combine the approaches in \cite{DKLP15} and \cite{ELN18}, and present a reduced fast SCS algorithm. This algorithm should be particularly well 
suited for situations 
in which one requires the construction of a large number of lattice points in high dimensions, with sufficiently fast decaying weights. The reduced fast SCS algorithm will again work by improving 
on a given starting vector $\bsz^0$, on a step-by-step basis (one component after the other). In comparison to the usual SCS algorithm presented in \cite{ELN18}, 
however, the search spaces for the single components of the 
output vector will be reduced according to the coordinate weights, thus speeding up the construction method. 
We are going to show that for suitable choices of weights the construction cost of the reduced fast SCS algorithm can be made independent of the dimension, and that its result can be at least as good as
that of the reduced fast CBC construction presented in \cite{DKLP15}. 
Our results will be shown for integration algorithms for functions in weighted Korobov spaces, but, as we shall see below, they also can be transferred to hold for certain Sobolev spaces of functions.
Apart from introducing and analyzing the reduced SCS algorithm for the construction of lattice points, our results imply a generalization of the results that have been presented in the paper 
\cite{ELN18}, in the sense that the SCS algorithm now works for $N$ being a prime power, and for general coordinate weights (as opposed to prime $N$ and product weights in \cite{ELN18}). 

We will also show that the SCS algorithm, as well as the reduced (fast) SCS algorithm can be adapted for constructing polynomial lattice rules which 
can be used for integrating functions in Walsh spaces and again certain Sobolev spaces. We stress that the present paper is the first paper where SCS algorithms for the polynomial lattice rule 
case are analyzed.

Moreover, we will present numerical results demonstrating that the reduced SCS algorithm constructs lattice rules which exhibit the
same error convergence rate as the (reduced) CBC construction provided the weights decay sufficiently fast. Additionally, we will demonstrate the speed of the reduced SCS algorithm via timings. This achieved speed-up in the construction of lattice rules is of
great importance when considering very high-dimensional integration problems as in, e.g., \cite{DKLGNS14} and \cite{KSS12},
and thus promising for further application.

\bigskip

The rest of the paper is structured as follows. In Section \ref{secKor}, we introduce Korobov spaces and point out how results for these are 
related to results for Sobolev spaces. Section \ref{secSCS} contains our main results regarding the reduced SCS construction for lattice rules. 
This is followed by remarks on how to obtain a fast implementation of the reduced SCS construction in Section \ref{secFastscs} and 
numerical results for lattice rules in Section \ref{secNum}. We conclude the paper with a section on corresponding results for polynomial lattice rules.

\section{Korobov spaces and related Sobolev spaces} \label{secKor}

We consider a weighted Korobov space with general weights as studied in
\cite{DSWW06,NW10}. Let us first introduce some notation. We denote by
$\ZZ$ the set of integers, by $\ZZ_{\ast}$ the set of integers excluding 0, and by $\NN$ the set of positive integers. As above, for $s \in \NN$ we
write $[s]=\{1,2,\ldots,s\}$. For a vector $\bsx=(x_1,\ldots,x_s)\in [0,1]^s$ and for $\setu \subseteq [s]$, we write $\bsx_\setu=(x_j)_{j \in \setu} \in
[0,1]^{|\setu|}$ and $(\bsx_{\setu},\bszero)\in [0,1]^s$ for the vector $(y_1,\ldots,y_s)$ with $y_j=x_j$ if $j \in \setu$ and $y_j=0$ if $j \not\in
\setu$. For integer vectors $\bsh\in\ZZ^s$, and $\setu\subseteq [s]$, we analogously write $\bsh_{\setu}$ to denote the projection of $\bsh$ onto those components with indices in $\setu$.

As outlined in the introduction, 
the importance of the different components or groups of components of the functions from the Korobov space to be defined is specified by a sequence of 
positive weights $\bsgamma=(\gamma_{\setu})_{\setu \subseteq [s]}$, where we 
may assume that $\gamma_{\emptyset}=1$. The smoothness of the functions in the space is described with a parameter $\alpha>1$. 

The weighted Korobov space, denoted by $\cH(K_{s,\alpha,\bsgamma})$, is a reproducing kernel Hilbert space with kernel function
\begin{align*}
	K_{s,\alpha,\bsgamma}(\bsx,\bsy) & = 1+\sum_{\emptyset \not=\setu \subseteq [s]} \gamma_{\setu} \prod_{j \in \setu}\left(\sum_{h \in \ZZ_{\ast}} \frac{\exp(2 \pi \icomp h(x_j-y_j))}{|h|^{\alpha}}\right)\\
	&= 1+ \sum_{\emptyset \not=\setu \subseteq [s]} \gamma_{\setu} \sum_{\bsh_{\setu}\in \ZZ_{\ast}^{|\setu|}} \frac{\exp(2 \pi \icomp \bsh_{\setu}\cdot (\bsx_{\setu}-\bsy_{\setu}))}{\prod_{j \in \setu}|h_j|^{\alpha}}.
\end{align*}
The corresponding inner product is
\[
\langle f,g\rangle_{K_{s,\alpha,\bsgamma}}=\sum_{\setu \subseteq [s]} \gamma_{\setu}^{-1} \sum_{\bsh_{\setu}\in \ZZ_{\ast}^{|\setu|}} \left(\prod_{j \in \setu}|h_j|^{\alpha}\right) \widehat{f}((\bsh_{\setu},\bszero)) \overline{\widehat{g}((\bsh_{\setu},\bszero))},
\]
where $\widehat{f}(\bsh)=\int_{[0,1]^s} f(\bst) \exp(-2 \pi \icomp \bsh \cdot \bst)\rd \bst$ is the $\bsh$-th Fourier coefficient of $f$.

For $h \in \ZZ_{\ast}$, we define $\rho_{\alpha}(h)=|h|^{-\alpha}$, and for $\bsh=(h_1,\ldots,h_s) \in \ZZ_{\ast}^s$ let 
$\rho_{\alpha}(\bsh)=\prod_{j=1}^s \rho_{\alpha}(h_j)$.

It is known (see, e.g., \cite{DSWW06}) that the squared worst-case error of a lattice rule generated by a vector 
$\bsz \in \ZZ^s$ in the weighted Korobov space $\cH(K_{s,\alpha,\bsgamma})$ is given by 
\begin{equation}\label{eqerrorexprlps}
e_{N,s}^2 (\bsz)=\sum_{\emptyset\neq\setu\subseteq [s]}\gamma_\setu 
\sum_{\bsh_{\setu}\in\D_\setu}\rho_\alpha (\bsh_\setu),
\end{equation}
where 
$$
\D_\setu:=\left\{\bsh_\setu\in\ZZ_{\ast}^{\abs{\setu}}\ : \ \bsh_\setu\cdot\bsz_\setu\equiv 0\ (\operatorname{mod}N) \right\}
$$
is called the dual lattice of the lattice generated by $\bsz$. In order to avoid too many parameters in the notation, we 
do not include the weights $\bsgamma$ when referring to the worst-case error $e_{N,s}$, unless this is essential for the context.

The worst-case error of lattice rules in a Korobov space can be related to the worst-case error in certain Sobolev spaces.
Indeed, consider a tensor product Sobolev space $\cH_{s,\bsgamma}^{\sob}$ of absolutely continuous functions whose mixed partial derivatives of order $1$ 
in each variable are square integrable, with norm (see \cite{H98}) 
\begin{equation*}
	\| f\|_{\cH_{s,\bsgamma}^{\sob}} = \left(\sum_{\setu \subseteq [s]}  \gamma_{\setu}^{-1} \int_{[0,1]^{|\setu|}} \left(\int_{[0,1]^{s-|\setu|}} 
	\frac{\partial^{|\setu|}}{\partial \bsx_{\setu}} f(\bsx) \rd \bsx_{[s]\setminus \setu} \right)^2 \rd \bsx_{\setu}\right)^{1/2},
\end{equation*}
where $\partial^{|\setu|}f/\partial \bsx_{\setu}$ denotes the mixed partial derivative with respect to all variables $j \in \setu$. 
As pointed out in \cite[Section~5]{DKS13}, the root mean square worst-case error $\widehat{e}_{N,s,\bsgamma}$ for QMC integration in 
$\cH_{s,\bsgamma}^{\sob}$ using randomly shifted lattice rules $(1/N)\sum_{k=0}^{N-1}f\left(\left\{ \frac{k}{N} \bsz+\bsDelta \right\} 
\right)$, i.e.,  
$$
\widehat{e}_{N,s,\bsgamma}(\bsz)=\left(\int_{[0,1)^s} e_{N,s,\bsgamma}^2(\bsz,\bsDelta) \rd \bsDelta\right)^{1/2},
$$ 
where $e_{N,s,\bsgamma}(\bsz,\bsDelta)$ is the worst-case error for 
QMC integration in $\cH_{s,\bsgamma}^{\sob}$ using a shifted integration lattice, 
is essentially the same as the worst-case error $e_{N,s,\bsgamma}^{(2)}$ in the weighted Korobov space 
$\cH(K_{s,2,\bsgamma})$ using the unshifted version of the lattice rules. In fact, we have 
\begin{equation} \label{eq:wceeqwce}
\widehat{e}_{N,s, 2 \pi^2 \bsgamma}(\bsz)=e_{N,s,\bsgamma}^{(2)}(\bsz),
\end{equation}
where $2 \pi^2 \bsgamma$ denotes the weights $( (2 \pi^2)^{|\setu|} \gamma_{\setu})_{\emptyset \not=\setu \subseteq [s]}$. 
For a connection to the so-called anchored Sobolev space see, e.g., \cite[Section~4]{HW00}. 

In a slightly different setting, the random shift can be replaced by the tent transform $\phi(x) = 1 - |1-2x|$ in each variable. 
For a vector $\bsx \in [0,1]^s$ let $\phi(\bsx)$ be defined component-wise. 
Let $\widetilde{e}_{N,s, \bsgamma}(\bsz)$ be the worst-case error in the unanchored weighted Sobolev space $\cH_{s, \bsgamma}^{\sob}$ 
using the QMC rule $(1/N)\sum_{k=0}^{N-1}f\left(\phi\left(\left\{ \frac{k}{N} \bsz \right\} \right) \right)$. 
Then it is known due to \cite{DNP14} and \cite{CKNS16} that
\begin{equation}\label{eq_wce_tent}
\widetilde{e}_{N,s, \pi^2 \bsgamma}(\bsz) \le e_{N,s,\bsgamma}^{(2)}(\bsz),
\end{equation} 
where $\pi^2 \bsgamma = (\pi^{2|\setu|} \gamma_{\setu})_{\emptyset \neq \setu \subseteq [s]}$, and that the CBC construction 
with quality criterion given by the worst-case error of the Korobov space $\cH(K_{s,2,\bsgamma})$ can be used to construct tent-transformed 
lattice rules which achieve the almost optimal convergence order in the space $\cH_{s,\pi^2 \bsgamma}^{\sob}$ under appropriate conditions 
on the weights $\bsgamma$ (see \cite[Corollary 1]{CKNS16}). Hence we also have a direct connection between integration in the Korobov space using lattice rules and integration in the unanchored Sobolev space using tent-transformed lattice rules.

Thus, the results that will be shown in the following are valid for the root mean square worst-case error and the worst-case error using tent-transformed lattice rules in the unanchored Sobolev space as well as for the worst-case error in the Korobov space. Hence it suffices to state them only for $e_{N,s}$. Equation~(\ref{eq:wceeqwce}) can be used to obtain results also for $\widehat{e}_{N,s,\bsgamma}$ and Equation~\eqref{eq_wce_tent} can be used to obtain results for $\widetilde{e}_{N,s,\bsgamma}$.

What is more, there is also a connection between the worst-case errors for numerical integration using polynomial lattice rules in the Walsh space that 
we will introduce in Section \ref{secWal} and the anchored \cite[Section~5]{DKPS05} and unanchored \cite[Section~6]{DP05} Sobolev space.

\section{The reduced successive coordinate search algorithm} \label{secSCS}

Let the number of quadrature points $N=b^m$ be a power of a prime number $b$, and $m \in \N$. Furthermore, we assume general weights $\gamma_{\setu}$, $\setu\subseteq [s]$. \\

We further assume we are given non-negative integers $w_j$ ordered in a non-decreasing fashion, i.e., 
$0 \le w_1 \le w_2 \le w_3 \le \cdots$ . Additionally, we set $s^*$ as the largest $j$ such that $w_j < m$. 
Next we define the reduced search space $\mathcal{Z}_{N,w_j}$ for the $j$-th component of the generating vector as
\begin{align*}
	\mathcal{Z}_{N,w_j} &= \begin{cases} \{z \in \{1,2,\ldots,b^{m-w_j}-1\} : \gcd(z,N)=1 \} & \mbox{if }  w_j < m,  \\ 
                                                              \{1\}  & \mbox{if } w_j \ge m, \end{cases} 
\end{align*}
and $Y_j = b^{w_j}$ for $j\in\{1,\ldots, s\}$. Then we consider the following algorithm for the construction of the generating vector $\bsz$ based on some initial vector $\bsz^0$. 

\begin{algorithm}\label{alg:lattice} 
Let $N=b^m$ be a prime power, let $\gamma_{\setu}$, $\setu\subseteq [s]$, be general weights, 
and let the worst-case error $e_{N,s}$ in the weighted Korobov space $\calH (K_{s,\alpha,\bsgamma})$ 
be defined as in Section \ref{secKor}. Let $w_1 \le w_2 \le \cdots \le w_s$ and $Y_j = b^{w_j}$ for $j\in\{1,\ldots, s\}$. 
Then we construct the generating vector $\bsz = (Y_1 z_1, \ldots, Y_s z_s)$ as 
follows.
\begin{itemize}
	\item \textbf{Input:} Starting vector $\bsz^0=(z_1^0,\ldots,z_s^0) \in \{0,1,\ldots,N-1\}^s$.
	\item For $d \in [s]$ assume $z_1,\ldots,z_{d-1}$ have already been selected. Then choose $z_d \in \mathcal{Z}_{N,w_d}$
	such that $e^2_{N,s}((Y_1 z_1, \ldots, Y_{d-1} z_{d-1}, Y_d z_d, z_{d+1}^0, \ldots, z_s^0))$ is minimized as a function of $z_d$.
	\item Increase $d$ until $z_1,\ldots,z_s$ are found.
\end{itemize}
\end{algorithm}

\begin{theorem} \label{thm: reduced SCS}
    Let the assumptions in Algorithm \ref{alg:lattice} hold. Let $\bsz = (Y_1 z_1, \ldots, Y_s z_s)$ be constructed by Algorithm \ref{alg:lattice} with initial
    vector $\bsz^0 \in \{0,1,\ldots,N-1\}^s$. Then, for $\lambda \in (\frac{1}{\alpha},1]$, the squared worst-case error $e^2_{N,s}(\bsz)$ satisfies
	\begin{align*}
		e^2_{N,s}((Y_1 z_1, \ldots, Y_s z_s))
		&\le
		\left( \sum_{d=1}^{s} \sum_{d \in \setu \subseteq [s]} \gamma_{\setu}^{\, \lambda} \frac{2 (2 \zeta(\alpha \lambda))^{|\setu|}}{b^{\max(0,m-w_d)}} \right)^{\frac{1}{\lambda}} .
	\end{align*}
\end{theorem}

\begin{proof}
	By \eqref{eqerrorexprlps}, we have for $\bsxi = (\xi_1, \ldots, \xi_s) \in \{1,\ldots,N-1\}^s$,
	\begin{align*}
		e^2_{N,s}(\bsxi)
		&=
		\sum_{\emptyset \ne \setu \subseteq [s]} \gamma_{\setu} \sum_{\bsh_{\setu} \in \mathcal{D}_{\setu}(\bsxi_{\setu})} \rho_{\alpha} (\bsh_{\setu})
	\end{align*}
	with $\mathcal{D}_{\setu}(\bsxi_{\setu}) = \{\bsh_{\setu} \in \Z_{*}^{|\setu|} : \bsh_{\setu} \cdot \bsxi_{\setu} \equiv 0\ (\operatorname{mod}N) \}$. We introduce the following notation:
	\begin{align*}
		g_{\setu}(\bsxi_{\setu}) := \gamma_{\setu} \sum_{\bsh_{\setu} \in \mathcal{D}_{\setu}(\bsxi_{\setu})} \rho_{\alpha} (\bsh_{\setu}) ,
		\qquad
		R_d(\bsxi) := \sum_{d \in \setu \subseteq [d]} g_{\setu}(\bsxi_{\setu}), 
	\end{align*}
	and hence we can rewrite the squared worst-case error as
	\begin{align*}
		e^2_{N,s}(\bsxi)
		&=
		\sum_{\emptyset \ne \setu \subseteq [s]} g_{\setu}(\bsxi_{\setu})
		=
		\sum_{d=1}^{s} R_d(\bsxi) . 	
	\end{align*}
	In the following we write, for $d \in \{1,\ldots,s\}$, $\bsz^{(d)}:=(Y_1 z_1, \ldots, Y_{d-1} z_{d-1}, Y_d z_d, z_{d+1}^0, \ldots, z_s^0)$.
	As minimizing $e^2_{N,s}(\bsz^{(d)})$ as a function of $z_d$ is equivalent to minimizing only those parts that depend on $z_d$, namely
	\begin{align*}
	\theta_d(\bsz^{(d)}) := \sum_{d \in \setu \subseteq [s]} g_{\setu}(\bsz_{\setu}^{(d)}),
	\end{align*}
	we consider this quantity for all $d$ and note that
	\begin{align*}
	e^2_{N,s}(\bsz)
	&=
	\sum_{d=1}^{s} R_d(Y_1 z_1, \ldots, Y_s z_s)
	=
	\sum_{d=1}^{s} \sum_{d \in \setu \subseteq [d]} g_{\setu}(\bsz_{\setu})
	\le
	\sum_{d=1}^{s} \sum_{d \in \setu \subseteq [s]} g_{\setu}(\bsz_{\setu}^{(d)})
	=
	\sum_{d=1}^{s} \theta_d(\bsz^{(d)}).
	\end{align*}
	We shall now make use of an inequality which is sometimes referred to as Jensen's
	inequality (see \cite{HLP34,J1906}):
	\[
	  \sum_{i=1}^M a_i
	  \le
	  \left(\sum_{i=1}^M a^p_i \right)^{1/p}
	  \qquad \text{for $0\le p\le 1$ and $a_1,\ldots,a_M \ge 0$}.
	\]
	Using Jensen's inequality we obtain, for $\lambda \in (\frac{1}{\alpha},1]$,
	\begin{align*}
	\left(e^2_{N,s}(\bsz)\right)^{\lambda}
	&\le
	\left(\sum_{d=1}^{s} \theta_d(\bsz^{(d)})\right)^{\lambda}
	\le 
	\sum_{d=1}^{s} \theta^{\lambda}_d(\bsz^{(d)}).
	\end{align*}
	By the standard averaging argument we obtain that, since the best choice for $z_d$ is at least as good as the average,
	\begin{eqnarray*}
	\theta^{\lambda}_d(\bsz^{(d)})&=&\theta^{\lambda}_d(Y_1 z_1, \ldots, Y_d z_d, z_{d+1}^0, \ldots, z_s^0)\\
	&\le&
	\frac{1}{|\mathcal{Z}_{N,w_d}|} \sum_{z \in \mathcal{Z}_{N,w_d}} \theta^{\lambda}_d(Y_1 z_1, \ldots, Y_{d-1} z_{d-1}, Y_d z, z_{d+1}^0, \ldots, z_s^0).
	\end{eqnarray*}
	We now use the notation $\hat{\bsz}^{(d)} = \hat{\bsz}^{(d)}(z) := (Y_1 z_1, \ldots, Y_{d-1} z_{d-1}, Y_d z, z_{d+1}^0, \ldots, z_s^0)$. 
	For the sake of readability, we will sometimes write $\hat{\bsz}^{(d)}=(\hat{z}_1,\ldots,\hat{z}_s)$ for short.
	Next, we establish an upper estimate for the quantity $\theta^{\lambda}_d(\hat{\bsz}^{(d)})$ for each $d \in \{1,\ldots,s\}$,
	\begin{align*}
		\theta^{\lambda}_d(\hat{\bsz}^{(d)})
		&=
		\left(\sum_{d \in \setu \subseteq [s]} g_{\setu}(\hat{\bsz}_{\setu}^{(d)})\right)^{\lambda}
		\le
		\sum_{d \in \setu \subseteq [s]} \gamma_{\setu}^{\lambda} \sum_{\bsh_{\setu} \in \mathcal{D}_{\setu}(\hat{\bsz}^{(d)}_{\setu})} \rho_{\alpha \lambda} (\bsh_{\setu}) \\
		&=
		\gamma_{\{d\}}^{\lambda} \sum_{h_d \in \mathcal{D}_{\{d\}}(Y_d z)} \rho_{\alpha \lambda} (h_d) 
		+ \sum_{\substack{\emptyset \ne \setv\subseteq [s] \\ d \notin \setv}} \gamma_{\setv \cup \{d\}}^{\lambda} \sum_{h_d \in \Z_{*}} \rho_{\alpha \lambda} (h_d)
	    \sum_{\substack{\bsh_{\setv} \in \Z_{*}^{|\setv|} \\ \sum_{j \in \setv} h_j \hat{z}_j \equiv -h_d Y_d z \; (N) }} \rho_{\alpha \lambda} (\bsh_{\setv}) ,
	\end{align*}
	where we used Jensen's inequality twice to obtain the first estimate, and where we write 
	$(N)$ to denote $\bmod\ N$ for short. This implies in turn that
	\begin{align*}
	\theta^{\lambda}_d(\bsz^{(d)}) \le \frac{1}{\abs{\mathcal{Z}_{N,w_d}}}\sum_{z\in\mathcal{Z}_{N,w_d}}\theta_d^{\lambda} (\hat{\bsz}^{(d)})
	\le
	T_1 + T_2 ,
	\end{align*}
	where
	\begin{align*}
		T_1 
		&= 
		\frac{1}{|\mathcal{Z}_{N,w_d}|} \sum_{z \in \mathcal{Z}_{N,w_d}} \gamma_{\{d\}}^{\lambda} 
		\sum_{h_d \in \mathcal{D}_{\{d\}}(Y_d z)} \rho_{\alpha \lambda} (h_d) 
	\end{align*}
	and
	\begin{align*}
	T_2 
	&= 
	\frac{1}{|\mathcal{Z}_{N,w_d}|} \sum_{z \in \mathcal{Z}_{N,w_d}} \sum_{\substack{\emptyset \ne \setv\subseteq [s] \\ d \notin \setv}} \gamma_{\setv \cup \{d\}}^{\lambda} \sum_{h_d \in \Z_{*}} \rho_{\alpha \lambda} (h_d) \sum_{\substack{\bsh_{\setv} \in \Z_{*}^{|\setv|} \\ \sum_{j \in \setv} h_j  \hat{z}_j \equiv -h_d Y_d z \; (N) }} \rho_{\alpha \lambda} (\bsh_{\setv}) . 
	\end{align*}
	For $T_1$ we consider the two possible cases $w_d \ge m$ and $w_d < m$. Then we obtain:
	\begin{itemize}
		\item If $w_d \ge m$, then $z \in \mathcal{Z}_{N,w_d}=\{1\}$ and $b^m = N$ is a divisor of $b^{w_d}$ so that
			\begin{align*}
				T_1 
				&= 
				\gamma_{\{d\}}^{\lambda} \sum_{\substack{h_d \in \Z_* \\ b^{w_d} h_d \equiv 0 \; (N)}} \rho_{\alpha \lambda} (h_d)
				=
				\gamma_{\{d\}}^{\lambda} \sum_{h_d \in \Z_*} \rho_{\alpha \lambda} (h_d)
				=
				\gamma_{\{d\}}^{\lambda} 2 \zeta(\alpha \lambda)
				=
				\gamma_{\{d\}}^{\lambda} \frac{2 \zeta(\alpha \lambda)}{b^{\max(0,m-w_d)}} .
			\end{align*}
		\item If $w_d < m$, then $h_d b^{w_d} z \equiv 0 \; (N)$, i.e., $b^{w_d} h_d z = k b^m$ for some $k \in \Z$, is equivalent to $h_d z = k b^{m-w_d}$
			for some $k \in \Z$. Now if $z \mid k \, b^{m-w_d}$ then, since $z \in \mathcal{Z}_{N,w_d}$, we have $b^\ell \hspace{-3.5pt}\not|\hspace{4pt} z$ for all $\ell =1,\ldots,m-w_d$ which implies that $z \mid k$, i.e., $k = k' z$ for some $k' \in \Z$. Hence
			\begin{align*}
				h_d b^{w_d} z \equiv 0 \; (N)
				&\Leftrightarrow					
				b^{w_d} h_d z = k b^m
				\Leftrightarrow
				h_d z = k b^{m-w_d}
				\Leftrightarrow
				h_d z = k' z b^{m-w_d} \\				
				&\Leftrightarrow	
				h_d = k' b^{m-w_d}
				\Leftrightarrow
				b^{m-w_d} \mid h_d,
			\end{align*}
			and we obtain
			\begin{align*}
				T_1 
				&= 
				\frac{1}{|\mathcal{Z}_{N,w_d}|} \sum_{z \in \mathcal{Z}_{N,w_d}} \gamma_{\{d\}}^{\lambda} \sum_{\substack{h_d \in \Z_* \\ Y_d h_d z \equiv 0 \; (N)}} \rho_{\alpha \lambda} (h_d)
				=
				\frac{1}{|\mathcal{Z}_{N,w_d}|} \sum_{z \in \mathcal{Z}_{N,w_d}} \gamma_{\{d\}}^{\lambda} \sum_{\substack{h_d \in \Z_* \\ b^{m-w_d} | h_d}} \rho_{\alpha \lambda} (h_d) \\
				&=
				\gamma_{\{d\}}^{\lambda} \sum_{h_d \in \Z_*} \rho_{\alpha \lambda} (b^{m-w_d} h_d)
				=
				\gamma_{\{d\}}^{\lambda} \sum_{h_d \in \Z_*} b^{-\alpha \lambda (m-w_d)}  \rho_{\alpha \lambda} (h_d) \\
				&= 
				\gamma_{\{d\}}^{\lambda} b^{-\alpha \lambda (m-w_d)} 2 \zeta(\alpha \lambda)
				\le
				\gamma_{\{d\}}^{\lambda} \frac{2 \zeta(\alpha \lambda)}{b^{\max(0,m-w_d)}} .
			\end{align*}
	\end{itemize}
	Therefore, in both possible cases, it holds that
	\begin{align*}
		T_1 
		&\le
		\gamma_{\{d\}}^{\lambda} \frac{2 \zeta(\alpha \lambda)}{b^{\max(0,m-w_d)}} .
	\end{align*}
	Similarly, we investigate the term $T_2$ for both cases.  		
	\begin{itemize}
		\item If $w_d \ge m$, then $z \in \mathcal{Z}_{N,w_d}=\{1\}$ and $b^{\max(0,m-w_d)} = b^0 = 1$, and so
		\begin{align*}
			T_2
			&=
			\sum_{\substack{\emptyset \ne \setv\subseteq [s] \\ d \notin \setv}} \gamma_{\setv \cup \{d\}}^{\lambda} \sum_{h_d \in \Z_{*}} \rho_{\alpha \lambda} (h_d) \sum_{\substack{\bsh_{\setv}\in \Z_{*}^{|\setv|} \\ \sum_{j \in \setv} h_j \hat{z}_j \equiv -h_d Y_d \; (N) }} \rho_{\alpha \lambda} (\bsh_{\setv}) \\
			&\le 
				\sum_{h_d \in \Z_{*}} \rho_{\alpha \lambda} (h_d) \sum_{\substack{\emptyset \ne \setv\subseteq [s] \\ d \notin \setv}} \gamma_{\setv \cup \{d\}}^{\lambda}  \sum_{\bsh_{\setv}\in \Z_{*}^{|\setv|}} \rho_{\alpha \lambda} (\bsh_{\setv}) \\
			&=
			\frac{2 \zeta(\alpha \lambda)}{b^{\max(0,m-w_d)}} \sum_{\substack{\emptyset \ne \setv\subseteq [s] \\ d \notin \setv}} \gamma_{\setv \cup \{d\}}^{\lambda} \sum_{\bsh_{\setv} \in \Z_{*}^{|\setv|}} \rho_{\alpha \lambda} (\bsh_{\setv}) \\
			&=
			\frac{2 \zeta(\alpha \lambda)}{b^{\max(0,m-w_d)}} \sum_{\substack{\emptyset \ne \setv\subseteq [s] \\ d \notin \setv}} \gamma_{\setv \cup \{d\}}^{\lambda} (2 \zeta(\alpha \lambda))^{|\setv|}
			=
			\sum_{\substack{\emptyset \ne \setv\subseteq [s] \\ d \notin \setv}} \gamma_{\setv \cup \{d\}}^{\lambda} \frac{(2 \zeta(\alpha \lambda))^{|\setv|+1}}{b^{\max(0,m-w_d)}} \\
			&=
			\sum_{\substack{\{d\} \ne \setu\subseteq [s] \\ d \in \setu}} \gamma_{\setu}^{\lambda} \frac{(2 \zeta(\alpha \lambda))^{|\setu|}}{b^{\max(0,m-w_d)}}.
		\end{align*}
		\item If $w_d < m$, then, with $\abs{\mathcal{Z}_{N,w_d}}=b^{m-w_d-1} (b-1)$, we write
		\begin{align*}
			T_2
			&=
			\frac{1}{b^{m-w_d-1} (b-1)} \left[ \sum_{z \in \mathcal{Z}_{N,w_d}} \sum_{\substack{\emptyset \ne \setv\subseteq [s] \\ d \notin \setv}} \gamma_{\setv \cup \{d\}}^{\lambda} \sum_{\substack{h_d \in \Z_{*} \\ h_d \equiv 0 \; (b^{m-w_d})}} \rho_{\alpha \lambda} (h_d) \sum_{\substack{\bsh_{\setv} \in \Z_{*}^{|\setv|} \\ \sum_{j \in \setv} h_j \hat{z}_j \equiv -h_d Y_d z \; (N) }} \rho_{\alpha \lambda} (\bsh_{\setv}) \right. \\
			&+ \left. \sum_{z \in \mathcal{Z}_{N,w_d}} \sum_{\substack{\emptyset \ne \setv\subseteq [s] \\ d \notin \setv}} \gamma_{\setv \cup \{d\}}^{\lambda} \sum_{\substack{h_d \in \Z_{*} \\ h_d \not\equiv 0 \; (b^{m-w_d})}} \rho_{\alpha \lambda} (h_d) \sum_{\substack{\bsh_{\setv} \in \Z_{*}^{|\setv|} \\ \sum_{j \in \setv} h_j \hat{z}_j \equiv -h_d Y_d z \; (N) }} \rho_{\alpha \lambda} (\bsh_{\setv}) \right] \\
			&=
			T_{2,1} + T_{2,2} 
			,
		\end{align*}
		where
		\begin{align*}
		    T_{2,1} = \frac{1}{b^{m-w_d-1} (b-1)} \sum_{z \in \mathcal{Z}_{N,w_d}} \sum_{\substack{\emptyset \ne \setv\subseteq [s] \\ d \notin \setv}} \gamma_{\setv \cup \{d\}}^{\lambda} \sum_{\substack{h_d \in \Z_{*} \\ h_d \equiv 0 \; (b^{m-w_d})}} \rho_{\alpha \lambda} (h_d) \sum_{\substack{\bsh_{\setv} \in \Z_{*}^{|\setv|} \\ \sum_{j \in \setv} h_j \hat{z}_j \equiv -h_d Y_d z \; (N) }} \rho_{\alpha \lambda} (\bsh_{\setv})
		\end{align*}
		and
		\begin{align*}
		 T_{2,2} = \frac{1}{b^{m-w_d-1} (b-1)} \sum_{z \in \mathcal{Z}_{N,w_d}} \sum_{\substack{\emptyset \ne \setv\subseteq [s] \\ d \notin \setv}} \gamma_{\setv \cup \{d\}}^{\lambda} \sum_{\substack{h_d \in \Z_{*} \\ h_d \not\equiv 0 \; (b^{m-w_d})}} \rho_{\alpha \lambda} (h_d) \sum_{\substack{\bsh_{\setv} \in \Z_{*}^{|\setv|} \\ \sum_{j \in \setv} h_j \hat{z}_j \equiv -h_d Y_d z \; (N) }} \rho_{\alpha \lambda} (\bsh_{\setv}).
		\end{align*}
		
		For $T_{2,1}$ we see that if $h_d \equiv 0 \; (b^{m-w_d})$ then $h_d Y_d z \equiv 0 \; (N)$. Thus
		\begin{align*}
			T_{2,1}
			&=
			\frac{1}{b^{m-w_d-1} (b-1)} \sum_{z \in \mathcal{Z}_{N,w_d}} \sum_{\substack{\emptyset \ne \setv\subseteq [s] \\ d \notin \setv}} \gamma_{\setv \cup \{d\}}^{\lambda} \sum_{\substack{h_d \in \Z_{*} \\ h_d \equiv 0 \; (b^{m-w_d})}} \rho_{\alpha \lambda} (h_d) \sum_{\substack{\bsh_{\setv} \in \Z_{*}^{|\setv|} \\ \sum_{j \in \setv} h_j \hat{z}_j \equiv 0 \; (N) }} \rho_{\alpha \lambda} (\bsh_{\setv}) \\
			&=
			\sum_{\substack{\emptyset \ne \setv\subseteq [s] \\ d \notin \setv}} \gamma_{\setv \cup \{d\}}^{\lambda} \sum_{h_d \in \Z_*} \rho_{\alpha \lambda} (b^{m-w_d} h_d) \sum_{\substack{\bsh_{\setv} \in \Z_{*}^{|\setv|} \\ \sum_{j \in \setv} 
			h_j \hat{z}_j \equiv 0 \; (N) }} \rho_{\alpha \lambda} (\bsh_{\setv}) \\
			&=
			\frac{2 \zeta(\alpha \lambda)}{(b^{m-w_d})^{\alpha \lambda}} \sum_{\substack{\emptyset \ne \setv\subseteq [s] \\ d \notin \setv}} \gamma_{\setv \cup \{d\}}^{\lambda} \sum_{\substack{\bsh_{\setv} \in \Z_{*}^{|\setv|} \\ \sum_{j \in \setv} 
			h_j \hat{z}_j \equiv 0 \; (N) }} \rho_{\alpha \lambda} (\bsh_{\setv}) \\
			&\le
			\frac{4 \zeta(\alpha \lambda)}{b^{m-w_d}} \sum_{\substack{\emptyset \ne \setv\subseteq [s] \\ d \notin \setv}} \gamma_{\setv \cup \{d\}}^{\lambda} \sum_{\substack{\bsh_{\setv} \in \Z_{*}^{|\setv|} \\ \sum_{j \in \setv} h_j \hat{z}_j \equiv 0 \; (N) }} \rho_{\alpha \lambda} (\bsh_{\setv}) .
		\end{align*}
		For $T_{2,2}$ we obtain
		\begin{eqnarray*}
			T_{2,2}
			&=&
			\frac{1}{|\calZ_{N,w_d}|} \sum_{z \in \mathcal{Z}_{N,w_d}} \sum_{\substack{\emptyset \ne \setv\subseteq [s] \\ d \notin \setv}} \gamma_{\setv \cup \{d\}}^{\lambda} \sum_{c=1}^{b^{m-w_d}-1} \sum_{\substack{h_d \in \Z_{*} \\ h_d \equiv -c z^{-1} \; (b^{m-w_d} )}} \rho_{\alpha \lambda} (h_d)
			\\
			& & \hspace{5.2cm}\times
			\sum_{\substack{\bsh_\setv\in \Z_{*}^{|\setv|} \\ \sum_{j\in \setv} h_j \hat{z}_j \equiv c Y_d \; (N)}} \rho_{\alpha \lambda} (\bsh_{\setv})
			,
		\end{eqnarray*}
		where $z^{-1}$ denotes the multiplicative inverse of $z$ in $\calZ_{N,w_d}$. For $c \in \{1,\ldots,b^{m-w_d}-1\}$ let $g:=\gcd(c,b^{m-w_d})$, then $\gcd\left(\tfrac{c}{g},\tfrac{b^{m-w_d}}{g}\right)=1$. Furthermore, note that 
		\begin{equation*}
			\{c \, z^{-1} \pmod{b^{m-w_d}} \ : \ z \in \calZ_{N,w_d} \}=\{c \, z \pmod{b^{m-w_d}} \ : \ z \in \calZ_{N,w_d} \}
			.
		\end{equation*}
		Hence, we obtain
		\begin{align*}
			&\sum_{z \in \mathcal{Z}_{N,w_d}} \sum_{\substack{h_d \in \Z_{*} \\ h_d \equiv -c z^{-1} \; (b^{m-w_d} )}} \rho_{\alpha \lambda} (h_d)
			= 
			\sum_{z \in \mathcal{Z}_{N,w_d}} \sum_{\substack{h_d \in \Z_{*} \\ h_d \equiv -c z \; (b^{m-w_d} )}} \frac{1}{|h_d|^{\alpha \lambda}} 
			\\
			&\quad= \sum_{z \in \mathcal{Z}_{N,w_d}} \sum_{k \in \ZZ} \frac{1}{|k b^{m-w_d} - c z|^{\alpha \lambda}}
			= 
			\frac{1}{g^{\alpha \lambda}} \sum_{z \in \mathcal{Z}_{N,w_d}} \sum_{k \in \ZZ} \frac{1}{|k (b^{m-w_d}/g)-(c/g) z|^{\alpha \lambda}}
			\\
			&\quad= \frac{1}{g^{\alpha \lambda}} \sum_{z \in \mathcal{Z}_{N,w_d}} \sum_{h \in \Z_{*} \atop h \equiv -(c/g) z \; (b^{m-w_d}/g) } 
			\frac{1}{|h|^{\alpha \lambda}}
			\le 
			\frac{g}{g^{\alpha \lambda}} \sum_{a=1}^{b^{m-w_d}/g -1} \sum_{h \in \Z_{*} \atop h \equiv a \; (b^{m-w_d}/g) } \frac{1}{|h|^{\alpha \lambda}}
		\end{align*}	
		and furthermore we have
		\begin{align*}
			&\sum_{a=1}^{b^{m-w_d}/g -1} \sum_{h \in \Z_{*} \atop h \equiv a \; (b^{m-w_d}/g) } \frac{1}{|h|^{\alpha \lambda}}
			= 
			\sum_{a=0}^{b^{m-w_d}/g -1} \sum_{h \in \Z_{*} \atop h \equiv a \; (b^{m-w_d}/g) } \frac{1}{|h|^{\alpha \lambda}} 
			- \sum_{h \in \Z_{*} \atop h \equiv 0 \; (b^{m-w_d}/g) } \frac{1}{|h|^{\alpha \lambda}}
			\\
			&\quad= 
			\sum_{h \in \Z_{*}} \frac{1}{|h|^{\alpha \lambda}} - \sum_{h \in \Z_{*}} \frac{1}{|h b^{m-w_d}/g|^{\alpha \lambda}}
			= 
			2 \zeta(\alpha \lambda) - \left(\frac{g}{b^{m-w_d}}\right)^{\alpha \lambda} 2 \zeta(\alpha \lambda)
			\le 
			2 \zeta(\alpha \lambda) 
			.
		\end{align*}
		Consequently, since $g \ge 1$ and $\lambda > 1/\alpha$, we have that
		\begin{equation*}
			\sum_{z \in \mathcal{Z}_{N,w_d}} \sum_{\substack{h_d \in \Z_{*} \\ h_d \equiv -c z^{-1} \; (b^{m-w_d} )}} \rho_{\alpha \lambda} (h_d)
			\le 
			\frac{g}{g^{\alpha \lambda}} \, 2 \zeta(\alpha \lambda) \le 2 \zeta(\alpha \lambda)
			,
		\end{equation*}
		from which it follows that 
		\begin{align*}
			T_{2,2}
			&\le
			\frac{2 \zeta(\alpha \lambda)}{|\calZ_{N,w_d}|} \sum_{\substack{\emptyset \ne \setv\subseteq [s] \\ d \notin \setv}} \gamma_{\setv \cup \{d\}}^{\lambda} \sum_{c=1}^{b^{m-w_d}-1} \sum_{\substack{\bsh_\setv\in \Z_{*}^{|\setv|} \\ \sum_{j\in \setv} h_j \hat{z}_j \equiv c Y_d \; (N)}} \rho_{\alpha \lambda} (\bsh_{\setv})
			\\
			&\le
			\frac{2 \zeta(\alpha \lambda)}{|\calZ_{N,w_d}|} \sum_{\substack{\emptyset \ne \setv\subseteq [s] \\ d \notin \setv}} \gamma_{\setv \cup \{d\}}^{\lambda} \sum_{\substack{\bsh_\setv\in \Z_{*}^{|\setv|} \\ \sum_{j\in \setv} h_j \hat{z}_j \not\equiv 0 \; (N)}} \rho_{\alpha \lambda} (\bsh_{\setv})
			.
		\end{align*}
		Remember that $| \calZ_{N,w_d}| = b^{m-w_d-1} (b-1) \ge b^{m-w_d}/2$. Therefore we obtain
		\begin{align*}
			T_{2,2} 
			&\le
			\frac{4 \zeta(\alpha \lambda)}{b^{m-w_d}} \sum_{\substack{\emptyset \ne \setv\subseteq [s] \\ d \notin \setv}} \gamma_{\setv \cup \{d\}}^{\lambda} \sum_{\substack{\bsh_\setv\in \Z_{*}^{|\setv|} \\ \sum_{j\in \setv} h_j \hat{z}_j \not\equiv 0 \; (N)}} \rho_{\alpha \lambda} (\bsh_{\setv})
			\\
			&=
			\frac{4 \zeta(\alpha \lambda)}{b^{m-w_d}} \sum_{\substack{\emptyset \ne \setv\subseteq [s] \\ d \notin \setv}} \gamma_{\setv \cup \{d\}}^{\lambda} \left( \sum_{\bsh_\setv\in \Z_{*}^{|\setv|}} \rho_{\alpha \lambda} (\bsh_{\setv})
			- \sum_{\substack{\bsh_\setv\in \Z_{*}^{|\setv|} \\ \sum_{j\in \setv} h_j \hat{z}_j \equiv 0 \; (N)}} \rho_{\alpha \lambda} (\bsh_{\setv}) \right)
			.
		\end{align*}
		Thus, we obtain for $T_2$ that
		\begin{align*}
			T_2 
			&=
			T_{2,1} + T_{2,2}
			\le
			\frac{4 \zeta(\alpha \lambda)}{b^{m-w_d}} \sum_{\substack{\emptyset \ne \setv\subseteq [s] \\ d \notin \setv}} \gamma_{\setv \cup \{d\}}^{\lambda} \sum_{\bsh_{\setv} \in \Z_{*}^{|\setv|}} \rho_{\alpha \lambda} (\bsh_{\setv}) \\
			&=
			\sum_{\substack{\emptyset \ne \setv\subseteq [s] \\ d \notin \setv}} \gamma_{\setv \cup \{d\}}^{\lambda} 
			\frac{2 (2 \zeta(\alpha \lambda))^{|\setv| + 1}}{b^{\max(0,m-w_d)}}
			=
			\sum_{\substack{\{d\} \ne \setu\subseteq [s] \\ d \in \setu}} \gamma_{\setu}^{\lambda} 
			\frac{2 (2 \zeta(\alpha \lambda))^{|\setu|}}{b^{\max(0,m-w_d)}}.
		\end{align*}
	\end{itemize}
	Combining both cases, $T_2$ is always bounded by
	\begin{align*}
	T_2 
	&\le
	\sum_{\substack{\{d\} \ne \setu\subseteq [s] \\ d \in \setu}} \gamma_{\setu}^{\lambda} 
	\frac{2 (2 \zeta(\alpha \lambda))^{|\setu|}}{b^{\max(0,m-w_d)}} .
	\end{align*}
	Hence, for the quantity $\theta^{\lambda}_d(\bsz^{(d)})$ we see that
	\begin{align*}
	\theta^{\lambda}_d(\bsz^{(d)})
	&\le
	T_1 + T_2
	\le
	\gamma_{\{d\}}^{\lambda} \frac{2 \zeta(\alpha \lambda)}{b^{\max(0,m-w_d)}} + \sum_{\substack{\{d\} \ne \setu\subseteq [s] \\ d \in \setu}} \gamma_{\setu}^{\lambda} \frac{2 (2 \zeta(\alpha \lambda))^{|\setu|}}{b^{\max(0,m-w_d)}}
	\le
	\sum_{d \in \setu\subseteq [s]} \gamma_{\setu}^{\lambda} \frac{2 (2 \zeta(\alpha \lambda))^{|\setu|}}{b^{\max(0,m-w_d)}},
	\end{align*}
	and so the squared worst-case error is bounded by
	\begin{align*}
	\left(e^2_{N,s}(\bsz)\right)^{\lambda} 
	&\le
	\sum_{d=1}^{s} \theta^{\lambda}_d(\bsz^{(d)})
	\le
	\sum_{d=1}^{s} \sum_{d \in \setu\subseteq [s]} \gamma_{\setu}^{\lambda} \frac{2 (2 \zeta(\alpha \lambda))^{|\setu|}}{b^{\max(0,m-w_d)}}
	\end{align*}
	which proves the claim. 
\end{proof}

\begin{corollary}\label{corlatticegenweights}
        Let the assumptions in Algorithm \ref{alg:lattice} hold.
        Let $\bsz = (Y_1 z_1,\ldots,Y_s z_s)$ be constructed by Algorithm \ref{alg:lattice}. 
	Then we have for all $\delta \in (0,\frac{\alpha-1}{2}]$ that
	\begin{align*}
		e_{N,s}(\bsz)
		\le
		C_{s,\alpha,\bsgamma,\delta} \; N^{-\alpha/2+\delta},
	\end{align*}
	where
	\begin{align*}
		C_{s,\alpha,\bsgamma,\delta}
		&:=
		\left(2\sum_{d=1}^{s} \sum_{d \in \setu\subseteq [s]} 
		\gamma_{\setu}^{\frac{1}{\alpha-2\delta}} 
		\left(2 \zeta\left(\frac{\alpha}{\alpha-2\delta}\right)\right)^{|\setu|} b^{w_d}\right)^{\alpha/2-\delta}.
	\end{align*}
        For $\delta\in (0,\frac{\alpha-1}{2}]$ and $q\ge 0$, define 
	\[
	 C_{\delta,q}:=\sup_{s\in\NN} \left[\frac{2}{s^q}\sum_{d=1}^{s}\sum_{d \in \setu\subseteq [s]} 
		\gamma_{\setu}^{\frac{1}{\alpha-2\delta}} \left(2 \zeta\left(\frac{\alpha}{\alpha-2\delta}\right)\right)^{|\setu|} b^{w_d} \right].
	\]
	If $C_{\delta,q}<\infty$ for some $\delta\in (0,\frac{\alpha-1}{2}]$ and $q\ge 0$ then
        \[
         e_{N,s}(\bsz)\le (s^q C_{\delta,q})^{\alpha/2-\delta} N^{-\alpha/2+\delta}.
        \]
	If $C_{\delta,0}<\infty$ for some $\delta\in (0,\frac{\alpha-1}{2}]$ then
        \[
         e_{N,s}(\bsz)\le (C_{\delta,0})^{\alpha/2-\delta} N^{-\alpha/2+\delta}.
        \]
\end{corollary}
\begin{proof}
	By Theorem \ref{thm: reduced SCS} we have, for $\lambda \in (\frac{1}{\alpha},1]$, that
	\begin{align*}
	\left(e^2_{N,s}(\bsz)\right)^{\lambda} 
	&\le
	\sum_{d=1}^{s} \sum_{d \in \setu\subseteq [s]} \gamma_{\setu}^{\lambda} \frac{2 (2 \zeta(\alpha \lambda))^{|\setu|}}{b^{\max(0,m-w_d)}}
	\end{align*}
	and thus, since $b^{-\max(0,m-w_d)} = b^{\min(0,w_d-m)} = b^{-m} \, b^{\min(m,w_d)} \le b^{-m} \, b^{w_d} = \frac1N \, b^{w_d}$,
	\begin{align*}
	\left(e^2_{N,s}(\bsz)\right)^{\lambda} 
	&\le
	\sum_{d=1}^{s} \sum_{d \in \setu\subseteq [s]} \gamma_{\setu}^{\lambda} \frac{2 (2 \zeta(\alpha \lambda))^{|\setu|}}{b^{\max(0,m-w_d)}}
	\le
	\frac{2}{N} \sum_{d=1}^{s} \sum_{d \in \setu\subseteq [s]} \gamma_{\setu}^{\lambda} (2 \zeta(\alpha \lambda))^{|\setu|} b^{w_d}.
	\end{align*}
	Setting $\frac{1}{\lambda}=\alpha-2\delta$, this shows the first assertion in the corollary. The proof of the further assertions is straightforward.
\end{proof}

Let us, for the next corollary, assume that we have product weights, i.e., $\gamma_{\setu}=\prod_{j\in\setu} \gamma_j$ for $\setu\subseteq [s]$, 
where the $\gamma_j$ are elements of an infinite, non-increasing sequence of positive reals, $(\gamma_j)_{j\ge 1}$.

\begin{corollary}\label{corlatticeproduct}
	Let the assumptions in Algorithm \ref{alg:lattice} hold.
	Let $\bsz = (Y_1 z_1,\ldots,Y_s z_s)$ be constructed by Algorithm \ref{alg:lattice}. 
	Then we have for all $\delta \in (0,\frac{\alpha-1}{2}]$ that
	\begin{align*}
		e_{N,s}(\bsz)
		\le
		C_{s,\alpha,\bsgamma,\delta} \; N^{-\alpha/2 + \delta},
	\end{align*}
	where
	\begin{align*}
		C_{s,\alpha,\bsgamma,\delta}
		&:=
		\left( \left( \sum_{d=1}^{s} \gamma_d^{\frac{1}{\alpha - 2 \delta}} b^{w_d} \right) \left( 4 \zeta\left(\frac{\alpha}{\alpha - 2 \delta}\right) \right) \prod_{d=1}^{s-1} \left( 1 +  \gamma_d^{\frac{1}{\alpha - 2 \delta}} 2  \zeta\left(\frac{\alpha}{\alpha - 2 \delta}\right) \right) \right)^{\alpha/2 - \delta} .
	\end{align*}
	Furthermore, the constant $C_{s,\alpha,\bsgamma,\delta}$ is bounded independently of the dimension $s$ if
	\begin{align*}
	\sum_{d=1}^{\infty} \gamma_d^{\frac{1}{\alpha - 2 \delta}} \, b^{w_d} < \infty .
	\end{align*}
\end{corollary}

\begin{proof}
	Similar to the proof of Corollary \ref{corlatticegenweights}, we see that
	\begin{align*}
	 \left(e^2_{N,s}(\bsz)\right)^{\lambda} 
	&\le
	\frac{2}{N} \sum_{d=1}^{s} \sum_{d \in \setu\subseteq [s]} \gamma_{\setu}^{\lambda} (2 \zeta(\alpha \lambda))^{|\setu|} b^{w_d}.
	\end{align*}
	Thus,	
	\begin{align*}
	\left(e^2_{N,s}(\bsz)\right)^{\lambda} 
	&\le
	\frac2N \sum_{d=1}^{s} \left( \sum_{\setu \subseteq [s] \setminus\{d\}} \gamma_{\setu}^{\lambda} (2 \zeta(\alpha \lambda))^{|\setu|} \right) \left(\gamma_d^{\lambda} 2 \, \zeta(\alpha \lambda) \, b^{w_d} \right) \\
	&\le
	\frac2N \sum_{d=1}^{s} \left(\gamma_d^{\lambda} \, b^{w_d} \right) (2 \, \zeta(\alpha \lambda)) \max_{d = 1,\ldots,s} \left( \sum_{\setu \subseteq [s] \setminus\{d\}} \gamma_{\setu}^{\lambda} (2 \zeta(\alpha \lambda))^{|\setu|} \right) \\
	&=
	\frac2N \sum_{d=1}^{s} \left(\gamma_d^{\lambda} \, b^{w_d} \right) (2 \, \zeta(\alpha \lambda)) \max_{d = 1,\ldots,s} \left( \prod_{\substack{j=1 \\ j \ne d}}^{s} \left(1 + \gamma_j^{\lambda} 2 \zeta(\alpha \lambda) \right) \right) \\
	&=
	\frac2N \sum_{d=1}^{s} \left(\gamma_d^{\lambda} \, b^{w_d} \right) (2 \, \zeta(\alpha \lambda))\prod_{j=1}^{s-1} \left(1 + \gamma_j^{\lambda} 2 \zeta(\alpha \lambda) \right) .
	\end{align*}
	Hence we have that 
	\begin{align*}
		e_{N,s}(\bsz) 
		&\le
		\left(\frac1N\right)^{\frac{1}{2\lambda}} \left(\sum_{d=1}^{s} \left(\gamma_d^{\lambda} \, b^{w_d} \right) (4 \, \zeta(\alpha \lambda))\prod_{j=1}^{s-1} 
		\left(1 + \gamma_j^{\lambda} 2 \zeta(\alpha \lambda) \right)\right)^{\frac{1}{2\lambda}},
	\end{align*}
	and setting $\frac1\lambda = \alpha - 2 \delta$ this gives
	\begin{align*}
	e_{N,s}(\bsz) 
	&\le
	N^{-\frac\alpha2 + \delta} \left(\sum_{d=1}^{s} \left(\gamma_d^{\frac{1}{\alpha - 2 \delta}} \, b^{w_d} \right) \left(4 \zeta\left(\frac{\alpha}{\alpha - 2 \delta}\right)\right) \prod_{j=1}^{s-1} \left(1 + \gamma_j^{\frac{1}{\alpha - 2 \delta}} 2 \zeta\left(\frac{\alpha}{\alpha - 2 \delta}\right) \right)\right)^{\alpha/2 - \delta} \\
	&=
	C_{s,\alpha,\bsgamma,\delta} \; N^{-\frac\alpha2 + \delta} .
	\end{align*}
	Furthermore, note that since
	\begin{align*}
		\prod_{j=1}^{s-1} \left(1 + \gamma_j^{\frac{1}{\alpha - 2 \delta}} 2 \zeta\left(\frac{\alpha}{\alpha - 2 \delta}\right) \right) =
		\exp\left(\log \left(\prod_{j=1}^{s-1} \left(1 + \gamma_j^{\frac{1}{\alpha - 2 \delta}} 2 \zeta\left(\frac{\alpha}{\alpha - 2 \delta}\right) \right)\right)\right),
	\end{align*}
	and (as $\log(1+x) \le x$)
	\begin{align*}
		\log \left(\prod_{j=1}^{s-1} \left(1 + \gamma_j^{\frac{1}{\alpha - 2 \delta}} 2 \zeta\left(\frac{\alpha}{\alpha - 2 \delta}\right) \right)\right)
		&=
		\sum_{d=1}^{s-1} \log \left(1 + \gamma_j^{\frac{1}{\alpha - 2 \delta}} 2 \zeta\left(\frac{\alpha}{\alpha - 2 \delta}\right) \right) \\
		&\le
		2 \zeta\left(\frac{\alpha}{\alpha - 2 \delta}\right) \sum_{d=1}^{s-1} \gamma_j^{\frac{1}{\alpha - 2 \delta}} \\
		&\le
		2 \zeta\left(\frac{\alpha}{\alpha - 2 \delta}\right) \sum_{d=1}^{\infty} \gamma_j^{\frac{1}{\alpha - 2 \delta}} b^{w_d} ,
	\end{align*}
	the constant $C_{s,\alpha,\bsgamma,\delta}$ is finite, and therefore bounded independently of the dimension $s$, if
	\begin{align*}
		\sum_{d=1}^{\infty} \gamma_j^{\frac{1}{\alpha - 2 \delta}} b^{w_d} < \infty .
	\end{align*}
\end{proof}

A straightforward but important consequence of Algorithm \ref{alg:lattice} and Theorem \ref{thm: reduced SCS} is that we obtain a generalization of one 
of the main results in \cite{ELN18}. In that paper, the (unreduced) SCS algorithm was considered for prime $N$ 
and for product weights. The following theorem generalizes this result to prime powers $N$ and to arbitrary weights.

\begin{theorem} \label{thm:SCSgeneral}
     Let $N=b^m$ be a prime power, let $\gamma_{\setu}$, $\setu\subseteq [s]$, be general weights, and let the worst-case error 
     $e_{N,s}$ in the weighted Korobov space $\calH (K_{s,\alpha,\bsgamma})$ be defined as in Section \ref{secKor}.
     Let $\bsz^{0}\in\{0,1,\ldots,N-1\}^s$ be an arbitrary initial vector. Then Algorithm \ref{alg:lattice} applied with $w_1=\cdots=w_s=0$
	constructs $\bsz = (z_1, \ldots, z_s)$ such that, for $\lambda \in (\frac{1}{\alpha},1]$,
	the squared worst-case error $e^2_{N,s}(\bsz)$ satisfies
	\begin{align*}
		e^2_{N,s}((z_1, \ldots, z_s))
		&\le
		\left( \sum_{\emptyset \ne \setu \subseteq  [s]} \abs{\setu} \gamma_{\setu}^{\, \lambda} \frac{2 (2 \zeta(\alpha \lambda))^{|\setu|}}{b^{m}} \right)^{\frac{1}{\lambda}}
		.
	\end{align*}
	In particular, $e_{N,s} (\bsz)\in \mathcal{O} (N^{-\alpha/2+\delta})$ for $\delta$ arbitrarily close to zero, where the implied constant is independent 
	of $s$ if
	\[
	 \sup_{s\in\NN} \left[2 \sum_{\emptyset \ne \setu\subseteq [s]} \abs{\setu}
	 \gamma_{\setu}^{\frac{1}{\alpha-2\delta}} \left(2 \zeta\left(\frac{\alpha}{\alpha-2\delta}\right)\right)^{|\setu|}  \right]
	 .
	\]
\end{theorem}
\begin{proof}
 The result follows immediately by considering Algorithm \ref{alg:lattice}, Theorem \ref{thm: reduced SCS}, 
 and Corollary \ref{corlatticegenweights} for the special case 
 $w_1=w_2=\cdots =w_d=0$. 
\end{proof}

\section{Fast SCS construction for product weights}	\label{secFastscs}

For product weights $\gamma_{\setu} = \prod_{j \in \setu} \gamma_j$ and $\bsxi=(\xi_1,\ldots,\xi_s) \in \{0,1,\ldots,N-1\}^s$, 
the squared worst-case error can be written as
\begin{align*}
	e^2_{N,s}(\bsxi) = -1 + \frac1N \sum_{k=0}^{N-1} \prod_{j=1}^{s} \left( 1 + \gamma_j \, \omega\left( \left\{ \frac{k \xi_j}{N} \right\} \right) \right) ,
\end{align*}
where $\omega$ is a real-valued function satisfying $\omega(x) = \omega(1-x)$ for $x\in [0,1]$, cf., e.g., \cite{KJ02} and \cite{K03}. For our purposes, we assume that the function $\omega$ can be evaluated in $N$ distinct arguments at a cost of at most $\mathcal{O}(N \log N)$; 
this assumption is justified for the setting studied in this paper, see, e.g., \cite{LP14}. Now, for one step of the reduced SCS algorithm with $w_d < m$, we need to find a component $z_d \in \mathcal{Z}_{N,w_d}$ such that $e^2_{N,s}((Y_1 z_1, \ldots, Y_{d-1} z_{d-1}, Y_d z_d, z_{d+1}^0, 
\ldots, z_s^0))$ is minimized as a function of $z_d$.
This is obviously equivalent to minimizing 
\begin{align*}
	\sum_{k=0}^{N-1} \left( 1 + \gamma_d \, \omega\left( \left\{ \frac{k Y_d z_d}{N} \right\} \right) \right) q_d(k)
	&=
	\sum_{k=0}^{N-1} q_d(k) + \gamma_d \sum_{k=0}^{N-1} \omega\left( \left\{ \frac{k Y_d z_d}{N} \right\} \right) q_d(k)
\end{align*}
as a function of $z_d$, where $Y_d = b^{w_d}$ and
\begin{align*}
	q_d(k) 
	&:=
	\left[\prod_{j=1}^{d-1} \left( 1 + \gamma_j \, \omega\left( \left\{ \frac{k Y_j z_j}{N} \right\} \right) \right)\right]
	\left[\prod_{j=d+1}^{s} \left( 1 + \gamma_j \, \omega\left( \left\{ \frac{k z_j^0}{N} \right\} \right) \right)\right]
	.
\end{align*}
Thus, the component $z_d$ is given by the $z \in \mathcal{Z}_{N,w_d}$ which minimizes
\begin{align*}
	T_d(z)
	&=
	\sum_{k=0}^{N-1} \omega\left( \frac{k b^{w_d} z \bmod N}{N} \right) q_d(k)
	.
\end{align*}
In the following we write $\ZZ_N$ to denote the set of integers $\{0,1,\ldots,N-1\}$. We note that $T_d(z)$ can be calculated simultaneously for all $z \in \mathcal{Z}_{N,w_d}$ as the matrix-vector product  
of the reduced matrix
\begin{align*}
	\Omega_{b^m,w}
	&:=
	\left[ \omega\left( \frac{k b^w z \bmod N}{N} \right) \right]_{\substack{z \in \mathcal{Z}_{N,w} \\ k \in \ZZ_N}}
	=
	\left[ \omega\left( \frac{k b^w z \bmod b^m}{b^m} \right) \right]_{\substack{z \in \mathcal{Z}_{N,w} \\ k \in \ZZ_{b^m}}}
\end{align*}
with $w = w_d$, and the vector $\bsq_d = (q_d(0),q_d(1),\ldots,q_d(N-1)) \in \RR^N$.

\subsection{The block-circulant structure of $\Omega_{b^m,w}$}

Due to the reduction of the search space from $\mathbb{U}_{b^m} = \{z \in \{1,2,\ldots,b^m-1\} : \gcd(z,b)=1\}$ to
\begin{align*}
	\mathcal{Z}_{N,w} 
	&=
	\{z \in \{1,2,\ldots,b^{m-w}-1\} : \gcd(z,b)=1 \}
	=
	\mathbb{U}_{b^{m-w}}
	,
\end{align*}
with $w < m$, the matrix $\Omega_{b^m,w}$ is of special block-circulant structure which allows a fast computation of the above matrix-vector product. The following two theorems, which will be shown in a combined proof, illustrate this structure for the cases $b\ne2$ and $b=2$, respectively.

In the following, for $t,r \ge 1$, we denote by $\langle \langle g \rangle \rangle_{b^r}$ the set $\{ g^i \bmod b^r \mid 0 \le i \le \frac{\varphi(b^r)}{2} - 1\}$,
and furthermore set $\bsone_{t} \otimes A$ and $\bsone_{t}^\top \otimes A$ as the vertical and horizontal stacking of $t$ instances of the matrix $A$, respectively.

\begin{theorem} \label{thm:matrix_structure_bnot2}
	For $b\ne2$, $w < m$, and $\omega:[0,1]\rightarrow\RR$ such that  $\omega(x) = \omega(1-x)$, the reduced matrix 
	\begin{align*}
	\Omega_{b^m,w}
	&:=
	\left[ \omega\left( \frac{k b^w z \bmod b^m}{b^m} \right) \right]_{\substack{z \in \mathcal{Z}_{N,w} \\ k \in \ZZ_{b^m}}}
	\end{align*}
	can, with respect to a generator $g$ of $\mathbb{U}_{b^m}$, be reordered to
	\begin{align*}
	\Omega^{\langle g \rangle}_{b^m,w}
	&:=
	\left[ 
	\mathbf{1}_{b^0} \otimes B^{\langle g \rangle}_{b^{m-w}} \mathrel{\Big|} 
	\mathbf{1}_{b^1} \otimes B^{\langle g \rangle}_{b^{m-w-1}} \mathrel{\Big|} \ldots \mathrel{\Big|}
	\mathbf{1}_{b^{m-w-1}} \otimes B^{\langle g \rangle}_{b^{1}} \mathrel{\Big|}
	\mathbf{1}_{b^w}^\top \otimes (\omega(0) \mathbf{1}_{\varphi(b^{m-w})})
	 \right]
	 ,
	\end{align*}
	where for $\ell \in \{w+1,w+2,\ldots,m\}$ and $r \in \{1,\ldots,m\}$ we define
	\begin{align*}	
	B^{\langle g \rangle}_{b^{\ell-w}}
	:=
	\left[ \bfrac{\mathbf{1}_{2 b^w}^\top \otimes M^{\langle g \rangle}_{b^{\ell-w}} }
	{\mathbf{1}_{2 b^w}^\top \otimes M^{\langle g \rangle}_{b^{\ell-w}} } \right] 
	\quad and \quad 
	M^{\langle g \rangle}_{b^{r}}
	:=
	\left[ \omega\left( \frac{k z \bmod b^r}{b^r}\right) \right]_
	{\substack{z \in \langle \langle g \rangle \rangle_{b^r} \\ k \in \langle \langle g^{-1} \rangle \rangle_{b^r}}}
	.
	\end{align*}
	Thus, $B^{\langle g \rangle}_{b^{\ell-w}}$, and therefore also $\Omega^{\langle g \rangle}_{b^m,w}$, consists of circulant blocks $M^{\langle g \rangle}_{b^{\ell-w}}$.
	
\end{theorem}

\begin{theorem} \label{thm:matrix_structure_b2}
	For $b=2$, $w < m$, and $\omega:[0,1]\rightarrow\RR$ such that  $\omega(x) = \omega(1-x)$, the reduced matrix 
	\begin{align*}
	\Omega_{2^m,w}
	&:=
	\left[ \omega\left( \frac{k 2^w z \bmod 2^m}{2^m} \right) \right]_{\substack{z \in \mathcal{Z}_{N,w} \\ k \in \ZZ_{2^m}}}
	\end{align*}
	can be reordered with respect to the divisors of $2^m$ and $g=5$ as
	\begin{align*}
	\Omega^{\langle g \rangle}_{2^m,w}
	&:=
	\left[ 
	\mathbf{1}_{2^0} \otimes B^{\langle g \rangle}_{2^{m-w}} \mathrel{\Big|}
	\mathbf{1}_{2^1} \otimes B^{\langle g \rangle}_{2^{m-w-1}} \mathrel{\Big|} \ldots \mathrel{\Big|}
	\mathbf{1}_{2^{m-w-2}} \otimes B^{\langle g \rangle}_{2^{2}} \right. \mathrel{\Big|} \\
	&\phantom{;= \left[ \right.} \left.
	\mathbf{1}_{2^w}^\top \otimes (\omega(1/2) \mathbf{1}_{2^{m-w-1}} ) \mathrel{\Big|}
	\mathbf{1}_{2^w}^\top \otimes (\omega(0) \mathbf{1}_{2^{m-w-1}})
	\right]
	,
	\end{align*}
	where for $\ell \in \{w+2,w+3,\ldots,m\}$ and $r \in \{1,\ldots,m\}$ we define
	\begin{align*}
	B^{\langle g \rangle}_{2^{\ell-w}}
	:=
	\left[ \bfrac{ \mathbf{1}_{2^{w+1}}^\top \otimes M^{\langle g \rangle}_{2^{\ell-w}} }
	{ \mathbf{1}_{2^{w+1}}^\top \otimes M^{\langle g \rangle}_{2^{\ell-w}} } \right] 
	\quad and \quad 
	M^{\langle g \rangle}_{2^{r}}
	:=
	\left[ \omega\left( \frac{k z \bmod 2^r}{2^r}\right) \right]_
	{\substack{z \in \langle \langle g \rangle \rangle_{2^r} \\ k \in \langle \langle g^{-1} \rangle \rangle_{2^r}}}
	.
	\end{align*}
	Thus, $B^{\langle g \rangle}_{2^{\ell-w}}$, and therefore also $\Omega^{\langle g \rangle}_{2^m,w}$, consists of circulant blocks $M^{\langle g \rangle}_{2^{\ell-w}}$.
\end{theorem}

\begin{proof}
	To prove Theorems \ref{thm:matrix_structure_bnot2} and \ref{thm:matrix_structure_b2} consider Theorems 4.2 and 4.3 in \cite{CKN06}
	which show how the unreduced matrix $\Omega_{b^m} = \Omega_{b^m,0}$ can be reordered with respect to the divisors of $b^m$ based on the circulant
	matrices $M^{\langle g \rangle}_{b^{r}}$. Since the matrix $\Omega_{b^m,w}$ can be obtained from $\Omega_{b^m}$ by replacing $k$ by $k b^w$ and only using the rows for which $z \in \mathbb{U}_{b^{m-w}}$, the matrix $\Omega_{b^m,w}$ inherits the structure of $\Omega_{b^m}$. This becomes evident by considering the above substitution for the circulant matrices $M^{\langle g \rangle}_{b^{r}}$. For $0 \le w < r$ we obtain that
	\begin{align} \label{substitution_matrix}
		\left[ \omega\left( \frac{k z b^w \bmod b^r}{b^r}\right) \right]_
		{\substack{z \in \langle \langle g \rangle \rangle_{b^r} \\ k \in \langle \langle g^{-1} \rangle \rangle_{b^r}}}
		&=
		\left[ \omega\left( \frac{k z \bmod b^{r-w}}{b^{r-w}}\right) \right]_
		{\substack{z \in \langle \langle g \rangle \rangle_{b^r} \\ k \in \langle \langle g^{-1} \rangle \rangle_{b^{r}}}} .	
	\end{align}
	Next, note that for $b=2$ and $b \ne 2$ the set $\mathbb{U}_{b^r}$ can be written as 
	\begin{align*}
		\mathbb{U}_{b^r}
		&=
		\langle \langle g \rangle \rangle_{b^r} \cup (-1) \langle \langle g \rangle \rangle_{b^r} ,
	\end{align*}
	where $(-1) \langle \langle g \rangle \rangle_{b^r} = \{-g^i \bmod b^r \mid 0 \le i \le \frac{\varphi(b^r)}{2} - 1\}$. For $b\ne 2$
	this follows from the fact that for the cyclic group $\mathbb{U}_{b^r}$ with generator $g$ we always have that $-1 \equiv g^{\varphi(b^r)/2}$.
	Hence, coming back to Equation \eqref{substitution_matrix}, we see that the two variables $z$ and $k$ iterate through the sets
	\begin{align*}
		\underbrace{
		\langle \langle a \rangle \rangle_{b^{r-w}} \cup (-1) \langle \langle a \rangle \rangle_{b^{r-w}} \cup \langle \langle a \rangle \rangle_{b^{r-w}} 
		\cup \ldots \cup \langle \langle a \rangle \rangle_{b^{r-w}} \cup (-1) \langle \langle a \rangle \rangle_{b^{r-w}}
		}_{\text{\footnotesize $b^w$ times}} 
	\end{align*}   
	for $a=g$ and $a=g^{-1}$, respectively.
	Thus, for $0 \le w \le r-2$, the matrices $M^{\langle g \rangle}_{b^{r}}$ with respect to the substitution $\tilde{k} = k b^w$ are given by
	\begin{align*}
		M^{\langle g \rangle}_{b^{r}}
		&=
		\left[ \omega\left( \frac{k z b^w \bmod b^r}{b^r}\right) \right]_
		{\substack{z \in \langle \langle g \rangle \rangle_{b^r} \\ k \in \langle \langle g^{-1} \rangle \rangle_{b^r}}}
		=
		\left[ \omega\left( \frac{k z \bmod b^{r-w}}{b^{r-w}}\right) \right]_
		{\substack{z \in \langle \langle g \rangle \rangle_{b^r} \\ k \in \langle \langle g^{-1} \rangle \rangle_{b^{r}}}} \\
		&=
		\underbrace{\left.
		\begin{bmatrix} 
		M^{\langle g \rangle}_{b^{r-w}} & M^{\langle g \rangle}_{b^{r-w}} & \cdots & M^{\langle g \rangle}_{b^{r-w}} \\
		M^{\langle g \rangle}_{b^{r-w}} & M^{\langle g \rangle}_{b^{r-w}} & \cdots &M^{\langle g \rangle}_{b^{r-w}} \\
		\vdots          & \vdots          & \ddots &\vdots \\
		M^{\langle g \rangle}_{b^{r-w}} & M^{\langle g \rangle}_{b^{r-w}} & \cdots &M^{\langle g \rangle}_{b^{r-w}} 
		\end{bmatrix}
		\right\}}_{\text{\footnotesize $b^w$ times}} \text{\footnotesize $b^w$ times}
		=
		\left.
		\begin{bmatrix} 
		\mathbf{1}_{b^{w}}^\top \otimes M^{\langle g \rangle}_{b^{r-w}} \\
		\mathbf{1}_{b^{w}}^\top \otimes M^{\langle g \rangle}_{b^{r-w}} \\
		\vdots \\
		\mathbf{1}_{b^{w}}^\top \otimes M^{\langle g \rangle}_{b^{r-w}}
		\end{bmatrix}
		\right\}\text{\footnotesize $b^w\,$times}
		,
	\end{align*}
	where the penultimate equality follows through the reasoning above and since $\omega(x) = \omega(1-x)$. The same statement holds true for
	$w = r-1$ and $b \ne 2$. For the case $b=2$ and $w = r-1$ we obtain a special case since then the above substitution yields
	\begin{align*}
		M^{\langle g \rangle}_{2^{r}}
		&=
		\left[ \omega\left( \frac{k z 2^{r-1} \bmod 2^r}{2^r}\right) \right]_
		{\substack{z \in \langle \langle g \rangle \rangle_{2^r} \\ k \in \langle \langle g^{-1} \rangle \rangle_{2^r}}}
		=
		\left[ \omega\left( \frac{k z \bmod 2}{2}\right) \right]_
		{\substack{z \in \langle \langle g \rangle \rangle_{2^r} \\ k \in \langle \langle g^{-1} \rangle \rangle_{2^r}}} \\
		&=
		\omega(1/2) \begin{bmatrix} 1 & 1 & \cdots & 1 \\ 1 & 1 & \cdots & 1 \\ \vdots & \vdots & \ddots &\vdots \\ 1 & 1 & \cdots &1 \end{bmatrix}
		\in \RR^{2^{r-2} \times 2^{r-2}} .
	\end{align*}
	For $w \ge r$ the matrix $M^{\langle g \rangle}_{b^{r}}$ reduces to
	\begin{align*}
	M^{\langle g \rangle}_{b^{r}}
	&=
	\left[ \omega\left( \frac{k z b^{w} \bmod b^r}{b^r}\right) \right]_
	{\substack{z \in \langle \langle g \rangle \rangle_{b^r} \\ k \in \langle \langle g^{-1} \rangle \rangle_{b^r}}}
	=
	\omega(0) \begin{bmatrix} 1 & 1 & \cdots & 1 \\ 1 & 1 & \cdots & 1 \\ \vdots & \vdots & \ddots &\vdots \\ 1 & 1 & \cdots &1 \end{bmatrix}
	\in \RR^{\frac{\varphi(b^r)}{2} \times \frac{\varphi(b^r)}{2}} .
	\end{align*} 
	For the special case $b=2$ there occurs the additional term $B_{2^1}^{\langle g \rangle} = [\omega(1/2)]$ in Thm. 4.3 of \cite{CKN06}, 
	however, for $w \ge 1$ and the substitution $\tilde{k} = k b^w$ this results in $\omega(0)$. Now, using the theorems in \cite{CKN06} 
	and putting all derived cases together we obtain the structure of $\Omega_{b^m,w}^{\langle g \rangle}$ as given in Theorems \ref{thm:matrix_structure_bnot2} and \ref{thm:matrix_structure_b2}. 
\end{proof}

Theorems \ref{thm:matrix_structure_bnot2} and \ref{thm:matrix_structure_b2} reveal that, due to the repetitive structure of the matrix
$\Omega_{b^m,w}$, in order to calculate $T_d(z)$ for all $z \in \calZ_{N,w_d}$, it is sufficient to calculate a matrix-vector product
with the smaller matrix
\begin{equation*}
	\Omega_{b^{m-w}}
	:=
	\left[ \omega\left( \frac{k z \bmod b^{m-w}}{b^{m-w}} \right) \right]_{\substack{z \in \calZ_{N,w} \\ k \in \Z_{b^{m-w}}}}
	,
\end{equation*}
where $w=w_d$. Based on this observation, we will formulate a fast version of Algorithm \ref{alg:lattice} in the next section.

\subsection{Computational complexity of the reduced SCS construction}

Firstly, denote by $s^{\ast}$ the largest integer such that $w_{s^{\ast}} < m$. In order to achieve a low 
computational complexity, we consider initial vectors $\bsz^0$ of the form
\begin{align} \label{form_initial_vector}
	\bsz^0
	&=
	(z_1^0,\ldots,z_s^0)
	=
	(Y_1 \bar{z}_1,\ldots,Y_s \bar{z}_s)
	\equiv
	(Y_1 \bar{z}_1,\ldots,Y_{s^{\ast}} \bar{z}_{s^{\ast}}, 0,\ldots,0) \bmod N
\end{align}
with $\bar{z}_j \in \mathcal{Z}_{N,w_j}$ for all $j \in \{1,\ldots,s\}$. The fast implementation of the reduced successive coordinate search algorithm can then be formulated as follows.

\begin{algorithm}[Reduced fast SCS algorithm] \mbox{} \label{alg:fast_red_scs}
\begin{enumerate}
	\item Precomputation:
	\begin{enumerate}[(a)]
		\item Compute $\omega\left(\frac{k}{b^m}\right)$ for $k=0,1,\ldots,b^m-1$ and store the results.
		\item For $\bsz^0$ as in (\ref{form_initial_vector}) and $k=0,1,\ldots,b^m-1$ initialize $\bsq=(q(0),\ldots,q(b^m-1))$ as
		\begin{align*}
			q(k)
			&:=
			\prod_{j=1}^{s} \left(1 + \gamma_j \, \omega\left(\frac{k z_j^0 \bmod b^m}{b^m}\right) \right).
		\end{align*}
		\item Set $d=1$ and $s^{\ast}$ to be the largest integer such that $w_{s^{\ast}} < m$.
	\end{enumerate}
\end{enumerate}
While $d \le \min\{s,s^{\ast}\}$:
\begin{enumerate}
\setcounter{enumi}{1}
	\item Set $\bsq_d$ via $\bsq$ by dividing out the initial choice $z_d^0$ (for $k=0,1,\ldots,b^m-1$) such that
	\begin{align*}
	q_d(k)
	&=
	\left[\prod_{j=1}^{d-1} \left( 1 + \gamma_j \, \omega\left( \left\{ \frac{k Y_j z_j}{N} \right\} \right) \right)\right]
	\left[\prod_{j=d+1}^{s} \left( 1 + \gamma_j \, \omega\left( \left\{ \frac{k z_j^0}{N} \right\} \right) \right)\right] .
	\end{align*}
	\item Partition the vector $\bsq_d$ into $b^{w_d}$ vectors $\bsq_d^{(1)},\ldots,\bsq_d^{(b^{w_d})}$ of length $b^{m - w_d}$, where
	\begin{align*}
		\bsq_d^{(\ell)} = (q_d(1 + (\ell-1)b^{m-w_d}),\ldots,q_d(\ell \, b^{m-w_d})) \quad \text{for} \quad \ell=1,\ldots,b^{w_d}
	\end{align*}
	and set $\bsq_d' = \bsq_d^{(1)} + \cdots + \bsq_d^{(b^{w_d})}$.
	\item Calculate $T_d(z) = \Omega_{b^{m-w_d}} \, \bsq'_d$ for all $z \in \mathcal{Z}_{N,w_d}$ using FFTs.
	\item Set $z_d = \arg\min_{z \in \mathcal{Z}_{N,w_d}} T_d(z)$.
	\item Update $\bsq$ via $\bsq_d$ by multiplying with the chosen $z_d$ (for $k=0,1,\ldots,b^m-1$) such that
	\begin{align*}
	q(k)
	&=
	\left[\prod_{j=1}^{d} \left( 1 + \gamma_j \, \omega\left( \left\{ \frac{k Y_j z_j}{N} \right\} \right) \right)\right]
	\left[\prod_{j=d+1}^{s} \left( 1 + \gamma_j \, \omega\left( \left\{ \frac{k z_j^0}{N} \right\} \right) \right)\right] .
	\end{align*}
	\item Increase $d$ by $1$.
\end{enumerate}
If $s > s^{\ast}$, then set $z_{s^{\ast}+1} = \cdots = z_s = 1$. The squared worst-case error is then given as
\begin{align*}
	e^2_{N,s}(Y_1 z_1,\ldots,Y_s z_s)
	&=
	-1 + \frac{1}{b^m} \sum_{k=0}^{b^m-1} q(k) .
\end{align*}
\end{algorithm}

\begin{theorem} \label{thm:computational_complexity_lat}
The computational complexity of Algorithm \ref{alg:fast_red_scs} is 
\begin{align*}
	\mathcal{O} \left( m b^m + \min\{s,s^{\ast}\} \, b^m + \sum_{d=1}^{\min\{s,s^{\ast}\}} (m-w_d) b^{m-w_d} \right).
\end{align*}
\end{theorem}

\begin{proof}
	The first term originates from the precalculation in (a) of Algorithm 2 which requires $\mathcal{O}(m b^m)$ operations.
	Due to the chosen form of initial vectors as in (\ref{form_initial_vector}), the initialization of $\bsq$ in (b) of Algorithm 2
	only requires $\mathcal{O}(\min\{s,s^{\ast}\} b^m)$ operations since for $k=0,1,\ldots,b^m-1$
	\begin{align*}
	q(k)
	&=
	\prod_{j=1}^{s^{\ast}} \left(1 + \gamma_j \, \omega\left(\frac{k z_j^0 \bmod b^m}{b^m}\right) \right)
	\prod_{j = s^{\ast} + 1}^{s} \left(1 + \gamma_j \, \omega(0) \right) .
	\end{align*}
	Furthermore, the updates for $\bsq_d$ and $\bsq$ in the Steps 2 and 6, respectively, can likewise be 
	done in $\mathcal{O}(\min\{s,s^{\ast}\} b^m)$ operations. The additions in Step 3 similarly require 
	$\mathcal{O}(\min\{s,s^{\ast}\} b^m)$ calculations. Lastly, the matrix-vector product in Step 4 can be
	computed in only $\mathcal{O}((m-w_d) b^{m-w_d})$ operations using FFTs (see, e.g., \cite{NC06b,NC06}) and the results of Theorems 
	\ref{thm:matrix_structure_bnot2} and \ref{thm:matrix_structure_b2}. This then gives the last
	term and proves the theorem.
\end{proof}

\begin{remark}
	Note that in the implementation of Algorithm \ref{alg:fast_red_scs}, the vector $\bsq$ also has
	to be ordered with respect to a generator $g$ as in Theorems \ref{thm:matrix_structure_bnot2}
	and \ref{thm:matrix_structure_b2} in order to exploit the special structure of the matrix $\Omega_{b^m,w}$.
\end{remark}

Furthermore, for initial vectors $\bsz^0$ as in (\ref{form_initial_vector}), we obtain the following useful
result.

\begin{theorem} \label{thm: z0-z}
	Let the initial vector $\bsz^0 \in \{0,1,\ldots,N-1\}^s$ be of the form (\ref{form_initial_vector}) and denote 
	by $\bsz$ the result of Algorithm \ref{alg:lattice} seeded with $\bsz^0$. Then the generating vector $\bsz$ 
	satisfies
	\begin{align*}
	e_{N,s}(\bsz) \le e_{N,s}(\bsz^0) ,
	\end{align*}
	i.e., the constructed vector $\bsz$ is always at least as good as the initial vector $\bsz^0$ with respect 
	to the associated worst-case error.
\end{theorem}

\begin{proof}
	For this special choice of initial vectors, the statement follows directly from the formulation of Algorithm
	\ref{alg:lattice}. Since the value of $z_d^0$ is in each minimization step $d \in [s]$ amongst the candidates
	for $z_d$ the worst-case error $e_{N,s}$ never grows.
\end{proof}

\section{Numerical results} \label{secNum}

In this section, the results from Sections \ref{secSCS} and \ref{secFastscs} which led to the reduced fast SCS construction, 
stated in Algorithm \ref{alg:fast_red_scs}, will be illustrated via numerical experiments. Here we consider the construction
of rank-1 lattices in weighted Korobov spaces $\cH(K_{s,\alpha,\bsgamma})$ of smoothness $\alpha > 1$, and, as in Section 
\ref{secFastscs}, we assume product weights $\gamma_{\setu} = \prod_{j \in \setu} \gamma_j$. For $\bsxi=(\xi_1,\ldots,\xi_s) \in \{0,1,\ldots,N-1\}^s$
the worst-case error is then given by 
\begin{align} \label{eq:wce_korobov_exp}
e^2_{N,s}(\bsxi) = -1 + \frac1N \sum_{k=0}^{N-1} \prod_{j=1}^{s} \left( 1 + \gamma_j \sum_{h \in \ZZ_{\ast}} \frac{\exp(2 \pi \icomp h k \xi_j/N)}{|h|^{\alpha}} \right) ,
\end{align}
and it is easy to check that the symmetry assumption which was previously imposed on $\omega$ is satisfied. For an even smoothness
parameter $\alpha$, the sum of exponentials in (\ref{eq:wce_korobov_exp}) simplifies to the Bernoulli polynomial $B_{\alpha}(\left\{k \xi_j/N\right\})$ modulo a constant, 
see, e.g., \cite{DKS13}. For ease of implementation, we will therefore restrict our experiments to the case $\alpha=2$ so that the worst-case error reads
\begin{align*}
e^2_{N,s}(\bsxi) = -1 + \frac1N \sum_{k=0}^{N-1} \prod_{j=1}^{s} \left( 1 + 2 \pi^2 \gamma_j B_2\left(\left\{\frac{k \xi_j}{N}\right\}\right) \right) .
\end{align*}
Due to the connection between Korobov and (unanchored) Sobolev spaces pointed out in Section \ref{secKor}, the presented results
remain also valid for integration in weighted Sobolev spaces using randomly shifted or tent-transformed lattice rules.

The subsequent sections are devoted to illustrating the key features of the reduced fast SCS algorithm, i.e., the error convergence rate 
of the constructed lattices, the computational complexity of the algorithm, and the precise worst-case errors.
In order to carry out a rigorous analysis, we will always compare the obtained results with those of the reduced and unreduced
CBC construction and the unreduced SCS construction as in \cite{ELN18}. The different algorithms have all been implemented
using Matlab R2016b.

\newpage

\subsection{Error convergence behavior} \label{subsec: error_convergence}

We consider the convergence rate of the worst-case error $e_{N,s}$ for different weight sequences $\bsgamma = (\gamma_j)_{j\ge 1}$
and reduction indices $w_j$ of the form $w_j = \lfloor c \log_b j \rfloor$ with $c>0$.  According to Corollary \ref{corlatticeproduct}, the almost optimal error convergence rate of $\mathcal{O}(N^{-1+\delta})$ for the reduced CBC and SCS algorithm will always be achieved for $N\To\infty$. Additionally, Corollary \ref{corlatticeproduct} implies that a constant independent of $s$ can be achieved provided that the chosen weights $\gamma_j$ satisfy
\begin{align} \label{ineq:condition}
\sum_{j=1}^{\infty} \gamma_j^{\frac{1}{2(1 - \delta)}} \, b^{w_j} 
&\le
\sum_{j=1}^{\infty} \gamma_j^{\frac{1}{2}} \, b^{w_j} 
<  
\infty .
\end{align}
It is to be expected that parameter choices which satisfy the condition in \eqref{ineq:condition} will yield a nicer error behavior also 
in numerical experiments, since the negative influence of high $s$ is not present anymore. In particular, if \eqref{ineq:condition} is satisfied, 
there should not be much difference in the error behavior of the vectors obtained by reduced and unreduced algorithms, respectively, since the 
negative influence of the $w_j$ washes away. Nevertheless, there are situations 
where the almost optimal convergence order $\mathcal{O}(N^{-1+\delta})$  is only visible for larger values of $N$ than those considered in our 
numerical experiments. In that sense, our numerical results are to be understood as illustrating a kind of ``pre-asymptotic'' error behavior.

Here, we consider two common types of weight sequences with the general form 
$\gamma_j = q ^j$ with $0<q<1$ or $\gamma_j = 1/j^a$ with $a > 1$. For the former type of weights, Corollary \ref{corlatticeproduct} 
assures the optimal error convergence rate, with constant independent of $s$, for any $q$. For the latter type, we see that since
\begin{align*}
\sum_{j=1}^{\infty} \gamma_j^{\frac{1}{2}} \, b^{w_j} 
&= 
\sum_{j=1}^{\infty} j^{-\frac{a}{2}} \, b^{\lfloor c \log_b j \rfloor}
\asymp
\sum_{j=1}^{\infty} j^{-\frac{a}{2}} \, b^{c \log_b j}
=
\sum_{j=1}^{\infty} j^{c-\frac{a}{2}},
\end{align*}
the convergence of the series on the right-hand side of \eqref{ineq:condition} 
is only guaranteed for small $\delta$ if $a > 2(1+c)$. In Figures 1 and 2 we display the results of numerical
experiments using different weights $\gamma_j$ for a moderate and rapid reduction of $w_j = \lfloor 2 \log_b j \rfloor$ and $\lfloor \frac72 \log_b j \rfloor$, respectively. 
The generating vectors $\bsz$ are constructed by the reduced and unreduced versions of both the CBC and the SCS construction,
where the initial vector for the reduced and unreduced SCS algorithm is fixed to $\bsz^0 = (Y_1,\ldots,Y_s)$ and $\bsz^0 = (1,\ldots,1)$, respectively.

\begin{figure}[H]
	\centering
	\textbf{Error convergence in the Korobov space with $s=100, \alpha=2, b=3, w_j=\lfloor 2 \log_b j \rfloor$.} \par\medskip 
	\hspace{-0.25cm}
	\centering
	\begin{subfigure}[b]{0.5\textwidth}
		\centering
		\begin{tikzpicture}
		\begin{axis}[%
		width=0.8\textwidth,
		height=0.8\textwidth,
		at={(0\textwidth,0\textwidth)},
		scale only axis,
		xmode=log,
		xmin=10,
		xmax=1000000,
		xminorticks=true,
		xlabel={Number of points $N=b^m$},
		xmajorgrids,
		ymode=log,
		ymin=1e-06,
		ymax=0.1,
		yminorticks=true,
		ylabel={Worst-case error $e_{N,s}(\mathbf{z})$},
		ymajorgrids,
		axis background/.style={fill=white},
		legend style={legend cell align=left,align=left,draw=white!15!black}
		]
		\addplot [color=mycolor1,solid,line width=0.7pt,mark=square,mark options={solid},forget plot]
		table[row sep=crcr]{%
			81	0.0190097600892007\\
			243	0.00700337390321026\\
			729	0.0026093752140544\\
			2187	0.000918569335043437\\
			6561	0.000332746702774257\\
			19683	0.000120532633844437\\
			59049	4.26219448127198e-05\\
			177147	1.51443163106254e-05\\
			531441	5.46655348643188e-06\\
		};
		\addplot [color=mycolor2,solid,line width=0.7pt,mark=triangle,mark options={solid},forget plot]
		table[row sep=crcr]{%
			81	0.0191306384093743\\
			243	0.00699825698376761\\
			729	0.00262531394230185\\
			2187	0.000920818725230272\\
			6561	0.000338340236986696\\
			19683	0.000119495106019267\\
			59049	4.36091768414625e-05\\
			177147	1.62677553945531e-05\\
			531441	5.8461188484379e-06\\
		};
		\addplot [color=mycolor3,solid,line width=0.7pt,mark=o,mark options={solid},forget plot]
		table[row sep=crcr]{%
			81	0.0305866026153664\\
			243	0.0107605206787417\\
			729	0.00378982989517774\\
			2187	0.00132168827206661\\
			6561	0.000467390007997536\\
			19683	0.000159730875939243\\
			59049	5.55916532438074e-05\\
			177147	1.95349563030929e-05\\
			531441	6.96729283754497e-06\\
		};
		\addplot [color=mycolor4,solid,line width=0.7pt,mark=diamond,mark options={solid},forget plot]
		table[row sep=crcr]{%
			81	0.0305881018884227\\
			243	0.0107810956467476\\
			729	0.0038890776574316\\
			2187	0.00135140765524873\\
			6561	0.000471745792936818\\
			19683	0.000166041867964049\\
			59049	5.86669143515538e-05\\
			177147	2.06230940412581e-05\\
			531441	7.19598041521424e-06\\
		};
		\addplot [color=mycolor5,dotted,line width=1.0pt]
		table[row sep=crcr]{%
			81	0.0611732052307328\\
			243	0.0215424971618203\\
			729	0.00758631466532811\\
			2187	0.00267156447876302\\
			6561	0.000940806844831773\\
			19683	0.000331310558408136\\
			59049	0.000116672924645162\\
			177147	4.10870435601593e-05\\
			531441	1.44690394420864e-05\\
		};
		\addlegendentry{$\mathcal{O}(N^{-0.95})$};
		\end{axis}
		\end{tikzpicture}
		\caption{Weight sequence $\bsgamma=(\gamma_j)_{j=1}^s$ with $\gamma_j = (0.2)^j$.}    
	\end{subfigure}
	\begin{subfigure}[b]{0.5\textwidth}  
		\centering 
		\begin{tikzpicture}
		\begin{axis}[%
		width=0.8\textwidth,
		height=0.8\textwidth,
		at={(0\textwidth,0\textwidth)},
		scale only axis,
		xmode=log,
		xmin=10,
		xmax=1000000,
		xminorticks=true,
		xlabel={Number of points $N=b^m$},
		xmajorgrids,
		ymode=log,
		ymin=0.01,
		ymax=100,
		yminorticks=true,
		ylabel={Worst-case error $e_{N,s}(\mathbf{z})$},
		ymajorgrids,
		axis background/.style={fill=white},
		legend style={legend cell align=left,align=left,draw=white!15!black}
		]
		\addplot [color=mycolor1,solid,line width=0.7pt,mark=square,mark options={solid},forget plot]
		table[row sep=crcr]{%
			81	7.98426741047185\\
			243	4.57893299147049\\
			729	2.61628192796676\\
			2187	1.48507202904465\\
			6561	0.839176237025462\\
			19683	0.470582587540824\\
			59049	0.261598212975844\\
			177147	0.144578943381338\\
			531441	0.0794061945283336\\
		};
		\addplot [color=mycolor2,solid,line width=0.7pt,mark=triangle,mark options={solid},forget plot]
		table[row sep=crcr]{%
			81	7.98717627118422\\
			243	4.57556840477024\\
			729	2.61872041677695\\
			2187	1.48566684074331\\
			6561	0.845572666495825\\
			19683	0.474777885499354\\
			59049	0.26484500713064\\
			177147	0.146560054727094\\
			531441	0.0807602071113651\\
		};
		\addplot [color=mycolor3,solid,line width=0.7pt,mark=o,mark options={solid},forget plot]
		table[row sep=crcr]{%
			81	8.92104287350194\\
			243	5.13712404150564\\
			729	2.93852435919298\\
			2187	1.6646238927478\\
			6561	0.939639742899799\\
			19683	0.52390958638307\\
			59049	0.291070577863984\\
			177147	0.159778778856683\\
			531441	0.0878698568253495\\
		};
		\addplot [color=mycolor4,solid,line width=0.7pt,mark=diamond,mark options={solid},forget plot]
		table[row sep=crcr]{%
			81	8.91602819503788\\
			243	5.1274487948461\\
			729	2.92329170693288\\
			2187	1.66185422333608\\
			6561	0.938753897794267\\
			19683	0.527351653875154\\
			59049	0.293721943355142\\
			177147	0.162991561620455\\
			531441	0.0891427302196082\\
		};
		\addplot [color=mycolor5,dotted,line width=1.0pt]
		table[row sep=crcr]{%
			81	17.8420857470039\\
			243	9.96715833112162\\
			729	5.56797263539268\\
			2187	3.11044715439901\\
			6561	1.73759501596879\\
			19683	0.970676012048478\\
			59049	0.542250588720187\\
			177147	0.302918478789712\\
			531441	0.169219926545111\\
		};
		\addlegendentry{$\mathcal{O}(N^{-0.53})$};
		\end{axis}
		\end{tikzpicture}
		\caption{Weight sequence $\bsgamma=(\gamma_j)_{j=1}^s$ with $\gamma_j = (0.8)^j$.}
	\end{subfigure}
	\vskip\baselineskip
	\hspace{-0.25cm}
	\centering
	\begin{subfigure}[b]{0.5\textwidth}
		\centering
		\begin{tikzpicture}
		\begin{axis}[%
		width=0.8\textwidth,
		height=0.8\textwidth,
		at={(0\textwidth,0\textwidth)},
		scale only axis,
		xmode=log,
		xmin=10,
		xmax=1000000,
		xminorticks=true,
		xlabel={Number of points $N=b^m$},
		xmajorgrids,
		ymode=log,
		ymin=1e-05,
		ymax=1,
		yminorticks=true,
		ylabel={Worst-case error $e_{N,s}(\mathbf{z})$},
		ymajorgrids,
		axis background/.style={fill=white},
		legend style={legend cell align=left,align=left,draw=white!15!black}
		]
		\addplot [color=mycolor1,solid,line width=0.7pt,mark=square,mark options={solid},forget plot]
		table[row sep=crcr]{%
			81	0.100444481016235\\
			243	0.0421670572710047\\
			729	0.0176348914759493\\
			2187	0.0071474035737722\\
			6561	0.002939449427362\\
			19683	0.00119471428242999\\
			59049	0.000481774151949303\\
			177147	0.00019467437936298\\
			531441	8.02118212774332e-05\\
		};
		\addplot [color=mycolor2,solid,line width=0.7pt,mark=triangle,mark options={solid},forget plot]
		table[row sep=crcr]{%
			81	0.10377879487224\\
			243	0.0422592834038276\\
			729	0.0171913155993051\\
			2187	0.00761965658619845\\
			6561	0.00309193487622798\\
			19683	0.00126631052607984\\
			59049	0.000508438846123129\\
			177147	0.000207548207797968\\
			531441	8.34361563807986e-05\\
		};
		\addplot [color=mycolor3,solid,line width=0.7pt,mark=o,mark options={solid},forget plot]
		table[row sep=crcr]{%
			81	0.193302807324251\\
			243	0.104425108240108\\
			729	0.0574554043225389\\
			2187	0.0304658036205289\\
			6561	0.0137716013829108\\
			19683	0.00465041588425034\\
			59049	0.00161836873222445\\
			177147	0.00054857877969643\\
			531441	0.000190978996444582\\
		};
		\addplot [color=mycolor4,solid,line width=0.7pt,mark=diamond,mark options={solid},forget plot]
		table[row sep=crcr]{%
			81	0.193304458832975\\
			243	0.105327142195581\\
			729	0.0575493432599634\\
			2187	0.0304791813024935\\
			6561	0.0137395737698882\\
			19683	0.00469600287688063\\
			59049	0.00160628500376854\\
			177147	0.00055286006639915\\
			531441	0.000194507495257577\\
		};
		\addplot [color=mycolor5,dotted,line width=1.0pt]
		table[row sep=crcr]{%
			81	0.193302807324251\\
			243	0.0802677625994521\\
			729	0.0333306784412835\\
			2187	0.0138403524575609\\
			6561	0.00574711842385578\\
			19683	0.0023864544114105\\
			59049	0.000990960032092001\\
			177147	0.000411489857299797\\
			531441	0.000170868548858777\\
		};
		\addlegendentry{$\mathcal{O}(N^{-0.80})$};
		\end{axis}
		\end{tikzpicture}
		\caption{Weight sequence $\bsgamma=(\gamma_j)_{j=1}^s$ with $\gamma_j = 1/j^3$.}
	\end{subfigure}
	\begin{subfigure}[b]{0.5\textwidth}  
		\centering 
		\begin{tikzpicture}		
		\begin{axis}[%
		width=0.8\textwidth,
		height=0.8\textwidth,
		at={(0\textwidth,0\textwidth)},
		scale only axis,
		xmode=log,
		xmin=10,
		xmax=1000000,
		xminorticks=true,
		xlabel={Number of points $N=b^m$},
		xmajorgrids,
		ymode=log,
		ymin=1e-06,
		ymax=0.1,
		yminorticks=true,
		ylabel={Worst-case error $e_{N,s}(\mathbf{z})$},
		ymajorgrids,
		axis background/.style={fill=white},
		legend style={legend cell align=left,align=left,draw=white!15!black}
		]
		\addplot [color=mycolor1,solid,line width=0.7pt,mark=square,mark options={solid},forget plot]
		table[row sep=crcr]{%
			81	0.0240585176237469\\
			243	0.00818001762519953\\
			729	0.00278539219469362\\
			2187	0.000936990159575669\\
			6561	0.000316351369633843\\
			19683	0.000106411433182955\\
			59049	3.6083453360805e-05\\
			177147	1.2161009192294e-05\\
			531441	4.12587294908273e-06\\
		};
		\addplot [color=mycolor2,solid,line width=0.7pt,mark=triangle,mark options={solid},forget plot]
		table[row sep=crcr]{%
			81	0.0240585176237607\\
			243	0.00817977002293763\\
			729	0.00278560614292497\\
			2187	0.000937323836234134\\
			6561	0.000316369633749238\\
			19683	0.000107366439105026\\
			59049	3.60834872058126e-05\\
			177147	1.21612465534906e-05\\
			531441	4.1354682341823e-06\\
		};
		\addplot [color=mycolor3,solid,line width=0.7pt,mark=o,mark options={solid},forget plot]
		table[row sep=crcr]{%
			81	0.0242397206583572\\
			243	0.00817556227483744\\
			729	0.0027793865803811\\
			2187	0.000945105903910769\\
			6561	0.000317864127254038\\
			19683	0.000107261923299374\\
			59049	3.60521669955764e-05\\
			177147	1.22196154443481e-05\\
			531441	4.11984096786143e-06\\
		};
		\addplot [color=mycolor4,solid,line width=0.7pt,mark=diamond,mark options={solid},forget plot]
		table[row sep=crcr]{%
			81	0.0242397206583527\\
			243	0.00817556227482386\\
			729	0.00277895111408751\\
			2187	0.000944778255369394\\
			6561	0.000317857235964381\\
			19683	0.000107312366371227\\
			59049	3.60746987449705e-05\\
			177147	1.22195245882031e-05\\
			531441	4.1185472510866e-06\\
		};
		\addplot [color=mycolor5,dotted,line width=1.0pt]
		table[row sep=crcr]{%
			81	0.0484794413167145\\
			243	0.0161598137722382\\
			729	0.00538660459074606\\
			2187	0.00179553486358202\\
			6561	0.000598511621194006\\
			19683	0.000199503873731335\\
			59049	6.65012912437785e-05\\
			177147	2.21670970812595e-05\\
			531441	7.38903236041983e-06\\
		};
		\addlegendentry{$\mathcal{O}(N^{-1})$};
		\end{axis}
		\end{tikzpicture}
		\caption{Weight sequence $\bsgamma=(\gamma_j)_{j=1}^s$ with $\gamma_j = 1/j^8$.}
	\end{subfigure}
	\vskip\baselineskip
	\begin{tikzpicture}
	\hspace{0.05\linewidth}
	\begin{customlegend}[legend columns=4,legend style={align=left,draw=none,column sep=2ex},legend entries={CBC,SCS,reduced CBC,reduced SCS}]
	\addlegendimage{color=mycolor1, mark=square,solid,line width=1.0pt,line legend}
	\addlegendimage{color=mycolor2, mark=triangle,solid,line width=1.0pt}  
	\addlegendimage{color=mycolor3, mark=o,solid,line width=1.0pt}
	\addlegendimage{color=mycolor4, mark=diamond,solid,line width=1.0pt}
	\end{customlegend}
	\end{tikzpicture}
		\caption{Convergence of the worst-case error $e_{N,s}(\bsz)$ in the weighted Korobov space $\cH(K_{s,\alpha,\bsgamma})$ of
			smoothness $\alpha=2$ with $s=100, b = 3$ and integer sequence $w_j = \lfloor 2 \log_b j \rfloor$. The generating vector $\bsz$
			is constructed via the reduced and unreduced CBC construction and the reduced and unreduced SCS algorithm, respectively.}  
\end{figure}

\begin{figure}[H]
	\centering
	\textbf{Error convergence in the Korobov space with $s=100, \alpha=2, b=3, w_j=\lfloor \frac72 \log_b j \rfloor$.} \par\medskip 
	\hspace{-0.25cm}
	\centering
	\begin{subfigure}[b]{0.5\textwidth}
		\centering
		\begin{tikzpicture}
		\begin{axis}[%
		width=0.8\textwidth,
		height=0.8\textwidth,
		at={(0\textwidth,0\textwidth)},
		scale only axis,
		xmode=log,
		xmin=10,
		xmax=1000000,
		xminorticks=true,
		xlabel={Number of points $N=b^m$},
		xmajorgrids,
		ymode=log,
		ymin=1e-06,
		ymax=1,
		yminorticks=true,
		ylabel={Worst-case error $e_{N,s}(\mathbf{z})$},
		ymajorgrids,
		axis background/.style={fill=white},
		legend style={legend cell align=left,align=left,draw=white!15!black}
		]
		\addplot [color=mycolor1,solid,line width=0.7pt,mark=square,mark options={solid},forget plot]
		table[row sep=crcr]{%
			81	0.0190097600892007\\
			243	0.00700337390321026\\
			729	0.0026093752140544\\
			2187	0.000918569335043437\\
			6561	0.000332746702774257\\
			19683	0.000120532633844437\\
			59049	4.26219448127198e-05\\
			177147	1.51443163106254e-05\\
			531441	5.46655348643188e-06\\
		};
		\addplot [color=mycolor2,solid,line width=0.7pt,mark=triangle,mark options={solid},forget plot]
		table[row sep=crcr]{%
			81	0.0191306384093743\\
			243	0.00699825698376761\\
			729	0.00262531394230185\\
			2187	0.000920818725230272\\
			6561	0.000338340236986696\\
			19683	0.000119495106019267\\
			59049	4.36091768414625e-05\\
			177147	1.62677553945531e-05\\
			531441	5.8461188484379e-06\\
		};
		\addplot [color=mycolor3,solid,line width=0.7pt,mark=o,mark options={solid},forget plot]
		table[row sep=crcr]{%
			81	0.107417407011684\\
			243	0.0498251469981294\\
			729	0.0180598562059244\\
			2187	0.00622807045034995\\
			6561	0.00208854945259394\\
			19683	0.00070291636976728\\
			59049	0.000235587361302276\\
			177147	7.90339735658898e-05\\
			531441	2.66151174745817e-05\\
		};
		\addplot [color=mycolor4,solid,line width=0.7pt,mark=diamond,mark options={solid},forget plot]
		table[row sep=crcr]{%
			81	0.107417407011684\\
			243	0.0498665356003917\\
			729	0.018077524909066\\
			2187	0.00624178944639091\\
			6561	0.00209939464735071\\
			19683	0.000702831318843465\\
			59049	0.000236776245113785\\
			177147	7.93219323129461e-05\\
			531441	2.68585318460725e-05\\
		};
		\addplot [color=mycolor5,dotted,line width=1.0pt]
		table[row sep=crcr]{%
			81	0.214834814023369\\
			243	0.0756553192513364\\
			729	0.0266424571689718\\
			2187	0.00938229500614999\\
			6561	0.00330402931772168\\
			19683	0.0011635329868874\\
			59049	0.000409744854355178\\
			177147	0.000144294014490879\\
			531441	5.08139697096624e-05\\
		};
		\addlegendentry{$\mathcal{O}(N^{-0.95})$};
		\end{axis}
		\end{tikzpicture}
		\caption{Weight sequence $\bsgamma=(\gamma_j)_{j=1}^s$ with $\gamma_j = (0.2)^j$.}    
	\end{subfigure}
	\begin{subfigure}[b]{0.5\textwidth}  
		\centering 
		\begin{tikzpicture}	
		\begin{axis}[%
		width=0.8\textwidth,
		height=0.8\textwidth,
		at={(0\textwidth,0\textwidth)},
		scale only axis,
		xmode=log,
		xmin=10,
		xmax=1000000,
		xminorticks=true,
		xlabel={Number of points $N=b^m$},
		xmajorgrids,
		ymode=log,
		ymin=0.01,
		ymax=100,
		yminorticks=true,
		ylabel={Worst-case error $e_{N,s}(\mathbf{z})$},
		ymajorgrids,
		axis background/.style={fill=white},
		legend style={legend cell align=left,align=left,draw=white!15!black}
		]
		\addplot [color=mycolor1,solid,line width=0.7pt,mark=square,mark options={solid},forget plot]
		table[row sep=crcr]{%
			81	7.98426741047185\\
			243	4.57893299147049\\
			729	2.61628192796676\\
			2187	1.48507202904465\\
			6561	0.839176237025462\\
			19683	0.470582587540824\\
			59049	0.261598212975844\\
			177147	0.144578943381338\\
			531441	0.0794061945283336\\
		};
		\addplot [color=mycolor2,solid,line width=0.7pt,mark=triangle,mark options={solid},forget plot]
		table[row sep=crcr]{%
			81	7.98717627118422\\
			243	4.57556840477024\\
			729	2.61872041677695\\
			2187	1.48566684074331\\
			6561	0.845572666495825\\
			19683	0.474777885499354\\
			59049	0.26484500713064\\
			177147	0.146560054727094\\
			531441	0.0807602071113651\\
		};
		\addplot [color=mycolor3,solid,line width=0.7pt,mark=o,mark options={solid},forget plot]
		table[row sep=crcr]{%
			81	15.5398596598874\\
			243	10.2164776252988\\
			729	6.31928613215737\\
			2187	3.90538908677031\\
			6561	2.32873422493935\\
			19683	1.3436632877236\\
			59049	0.744194528370983\\
			177147	0.401072040278803\\
			531441	0.20726166347435\\
		};
		\addplot [color=mycolor4,solid,line width=0.7pt,mark=diamond,mark options={solid},forget plot]
		table[row sep=crcr]{%
			81	15.5398596598874\\
			243	10.2085273306613\\
			729	6.27596884643545\\
			2187	3.88109568912166\\
			6561	2.28691533596137\\
			19683	1.33655534035704\\
			59049	0.735294670318077\\
			177147	0.396994146436193\\
			531441	0.208000438782122\\
		};
		\addplot [color=mycolor5,dotted,line width=1.0pt]
		table[row sep=crcr]{%
			81	31.0797193197747\\
			243	17.3621227775476\\
			729	9.69903570367328\\
			2187	5.41819078153163\\
			6561	3.02677423220085\\
			19683	1.69085265213295\\
			59049	0.944564236344175\\
			177147	0.52766371774322\\
			531441	0.294769787283312\\
		};
		\addlegendentry{$\mathcal{O}(N^{-0.53})$};
		\end{axis}
		\end{tikzpicture}
		\caption{Weight sequence $\bsgamma=(\gamma_j)_{j=1}^s$ with $\gamma_j = (0.8)^j$.}
	\end{subfigure}
	\vskip\baselineskip
	\hspace{-0.25cm}
	\centering
	\begin{subfigure}[b]{0.5\textwidth}
		\centering
		\begin{tikzpicture}
		\begin{axis}[%
		width=0.8\textwidth,
		height=0.8\textwidth,
		at={(0\textwidth,0\textwidth)},
		scale only axis,
		xmode=log,
		xmin=10,
		xmax=1000000,
		xminorticks=true,
		xlabel={Number of points $N=b^m$},
		xmajorgrids,
		ymode=log,
		ymin=1e-05,
		ymax=1,
		yminorticks=true,
		ylabel={Worst-case error $e_{N,s}(\mathbf{z})$},
		ymajorgrids,
		axis background/.style={fill=white},
		legend style={legend cell align=left,align=left,draw=white!15!black}
		]
		\addplot [color=mycolor1,solid,line width=0.7pt,mark=square,mark options={solid},forget plot]
		table[row sep=crcr]{%
			81	0.100444481016235\\
			243	0.0421670572710047\\
			729	0.0176348914759493\\
			2187	0.0071474035737722\\
			6561	0.002939449427362\\
			19683	0.00119471428242999\\
			59049	0.000481774151949303\\
			177147	0.00019467437936298\\
			531441	8.02118212774332e-05\\
		};
		\addplot [color=mycolor2,solid,line width=0.7pt,mark=triangle,mark options={solid},forget plot]
		table[row sep=crcr]{%
			81	0.10377879487224\\
			243	0.0422592834038276\\
			729	0.0171913155993051\\
			2187	0.00761965658619845\\
			6561	0.00309193487622798\\
			19683	0.00126631052607984\\
			59049	0.000508438846123129\\
			177147	0.000207548207797968\\
			531441	8.34361563807986e-05\\
		};
		\addplot [color=mycolor3,solid,line width=0.7pt,mark=o,mark options={solid},forget plot]
		table[row sep=crcr]{%
			81	0.4064129110936\\
			243	0.304618575995613\\
			729	0.213364881481604\\
			2187	0.159382712192986\\
			6561	0.110058147420463\\
			19683	0.0812112869419721\\
			59049	0.0569369251969369\\
			177147	0.0411207941927234\\
			531441	0.0285709253445273\\
		};
		\addplot [color=mycolor4,solid,line width=0.7pt,mark=diamond,mark options={solid},forget plot]
		table[row sep=crcr]{%
			81	0.406412911093599\\
			243	0.304618575995613\\
			729	0.213342994682836\\
			2187	0.15939351677634\\
			6561	0.110041766833969\\
			19683	0.0812053330581222\\
			59049	0.0569407479079832\\
			177147	0.0411207197445899\\
			531441	0.028571560107173\\
		};
		\addplot [color=mycolor5,dotted,line width=1.0pt]
		table[row sep=crcr]{%
			81	0.4064129110936\\
			243	0.168760379202836\\
			729	0.0700766752514067\\
			2187	0.0290988941686888\\
			6561	0.0120831309248444\\
			19683	0.00501744334683457\\
			59049	0.00208346147163995\\
			177147	0.000865144138906255\\
			531441	0.000359245607020849\\
		};
		\addlegendentry{$\mathcal{O}(N^{-0.80})$};
		\end{axis}
		\end{tikzpicture}
		\caption{Weight sequence $\bsgamma=(\gamma_j)_{j=1}^s$ with $\gamma_j = 1/j^3$.}
	\end{subfigure}
	\begin{subfigure}[b]{0.5\textwidth}  
		\centering 
		\begin{tikzpicture}
		\begin{axis}[%
		width=0.8\textwidth,
		height=0.8\textwidth,
		at={(0\textwidth,0\textwidth)},
		scale only axis,
		xmode=log,
		xmin=10,
		xmax=1000000,
		xminorticks=true,
		xlabel={Number of points $N=b^m$},
		xmajorgrids,
		ymode=log,
		ymin=1e-06,
		ymax=0.1,
		yminorticks=true,
		ylabel={Worst-case error $e_{N,s}(\mathbf{z})$},
		ymajorgrids,
		axis background/.style={fill=white},
		legend style={legend cell align=left,align=left,draw=white!15!black}
		]
		\addplot [color=mycolor1,solid,line width=0.7pt,mark=square,mark options={solid},forget plot]
		table[row sep=crcr]{%
			81	0.0240585176237469\\
			243	0.00818001762519953\\
			729	0.00278539219469362\\
			2187	0.000936990159575669\\
			6561	0.000316351369633843\\
			19683	0.000106411433182955\\
			59049	3.6083453360805e-05\\
			177147	1.2161009192294e-05\\
			531441	4.12587294908273e-06\\
		};
		\addplot [color=mycolor2,solid,line width=0.7pt,mark=triangle,mark options={solid},forget plot]
		table[row sep=crcr]{%
			81	0.0240585176237607\\
			243	0.00817977002293763\\
			729	0.00278560614292497\\
			2187	0.000937323836234134\\
			6561	0.000316369633749238\\
			19683	0.000107366439105026\\
			59049	3.60834872058126e-05\\
			177147	1.21612465534906e-05\\
			531441	4.1354682341823e-06\\
		};
		\addplot [color=mycolor3,solid,line width=0.7pt,mark=o,mark options={solid},forget plot]
		table[row sep=crcr]{%
			81	0.0287582141641256\\
			243	0.0101609116644318\\
			729	0.00352992686615639\\
			2187	0.00124199854905788\\
			6561	0.000429054263436056\\
			19683	0.000148019169431172\\
			59049	5.07079027570693e-05\\
			177147	1.73693281108935e-05\\
			531441	5.95677529845048e-06\\
		};
		\addplot [color=mycolor4,solid,line width=0.7pt,mark=diamond,mark options={solid},forget plot]
		table[row sep=crcr]{%
			81	0.0287582141641218\\
			243	0.010160911664399\\
			729	0.00352992686584188\\
			2187	0.00124307140434386\\
			6561	0.000429055438724709\\
			19683	0.000147998786072723\\
			59049	5.07197025074176e-05\\
			177147	1.7438108955795e-05\\
			531441	5.98147533569857e-06\\
		};
		\addplot [color=mycolor5,dotted,line width=1.0pt]
		table[row sep=crcr]{%
			81	0.0575164283282513\\
			243	0.0191721427760838\\
			729	0.00639071425869459\\
			2187	0.00213023808623153\\
			6561	0.000710079362077176\\
			19683	0.000236693120692392\\
			59049	7.8897706897464e-05\\
			177147	2.6299235632488e-05\\
			531441	8.766411877496e-06\\
		};
		\addlegendentry{$\mathcal{O}(N^{-1})$};
		\end{axis}
		\end{tikzpicture}
		\caption{Weight sequence $\bsgamma=(\gamma_j)_{j=1}^s$ with $\gamma_j = 1/j^8$.}
	\end{subfigure}
	\vskip\baselineskip
	\begin{tikzpicture}
	\hspace{0.05\linewidth}
	\begin{customlegend}[legend columns=4,legend style={align=left,draw=none,column sep=2ex},legend entries={CBC,SCS,reduced CBC,reduced SCS}]
	\addlegendimage{color=mycolor1, mark=square,solid,line width=1.0pt,line legend}
	\addlegendimage{color=mycolor2, mark=triangle,solid,line width=1.0pt}  
	\addlegendimage{color=mycolor3, mark=o,solid,line width=1.0pt}
	\addlegendimage{color=mycolor4, mark=diamond,solid,line width=1.0pt}
	\end{customlegend}
	\end{tikzpicture}
	\caption{Convergence of the worst-case error $e_{N,s}(\bsz)$ in the weighted Korobov space $\cH(K_{s,\alpha,\bsgamma})$ of
		smoothness $\alpha=2$ with $s=100, b = 3$ and integer sequence $w_j = \lfloor \frac72 \log_b j \rfloor$. The generating vector $\bsz$
		is constructed via the reduced and unreduced CBC construction and the reduced and unreduced SCS algorithm, respectively.}  
\end{figure}

\newpage

As expected, Figures 1 and 2 illustrate that for weight sequences of geometric decay the convergence order is the same for the reduced and unreduced algorithms (see Cases (a) and (b) in both figures). Furthermore, note that the weights $\gamma_j = 1/j^3$ do not satisfy condition \eqref{ineq:condition} for any of the chosen $w_j$ such that, as we expected, the pre-asymptotic convergence order displayed by the unreduced CBC and SCS constructions is better than that of the reduced CBC and SCS constructions. This becomes evident by considering Case (c) in Figures 1 and 2. Note that the choice of $\gamma_j = 1/j^8$ and $w_j = \lfloor \frac72 \log_b j \rfloor$ also does not satisfy $a > 2(1+c)$, however, we still observe almost no difference between the reduced and the unreduced algorithms (see Case (d) in Figure 2). As can be seen from Figures 1 and 2, the error rates of the reduced and unreduced version of the algorithms are the same if the condition from Corollary \ref{corlatticeproduct} holds, but they differ by a multiplicative constant. The larger the $w_j$, the larger this multiplicative constant is (compare between Figures 1 and 2). An explanation for this observation is given by identifying the observed constant with the constant $C_{s,\alpha,\bsgamma,\delta}$ of Corollary \ref{corlatticeproduct}.

\subsection{Timings for the reduced fast SCS algorithm}

Here, we illustrate the computational complexity of the reduced fast SCS construction in Algorithm \ref{alg:fast_red_scs} which was stated 
in Theorem \ref{thm:computational_complexity_lat}. For that purpose, let $b=2$ and $N=b^m$ and let the weight sequence 
$\bsgamma = (\gamma_j)_{j \ge 1}$ be given by $\gamma_j = (0.7)^j$. Note that the choice of the weights $\gamma_j$ does not
influence the construction cost of the considered algorithms. In Tables \ref{tab:CBC_sequence1}, \ref{tab:SCS_sequence1} and  
Tables \ref{tab:CBC_sequence2}, \ref{tab:SCS_sequence2} below, we report on the computation times for the construction 
of the generating vector $\bsz$ via the four considered algorithms for the two reductions given by $w_j = \lfloor \frac32 \log_b j \rfloor$ 
and $w_j = \lfloor 3 \log_b j \rfloor$, respectively. Again, the two SCS algorithms (cf. Tables \ref{tab:SCS_sequence1} and
\ref{tab:SCS_sequence2}) are seeded with initial vectors $\bsz^0 = (Y_1,\ldots,Y_s)$ and $\bsz^0 = (1,\ldots,1)$, respectively. 
We emphasize that the used algorithms solely construct the generating vector $\bsz$ but do not calculate the worst-case 
error $e_{N,s}(\bsz)$, which allows for an unbiased comparison between the considered algorithms. The computations and timings
were performed on an Intel Core i5-2400S CPU with 2.5GHz using Matlab.

\begin{table}[H]
	\captionof{table}{Computation times (in seconds) for constructing the generating vector $\bsz$ using the unreduced (normal font) and reduced CBC (\textbf{bold font}) construction. 
	The associated lattice can be used for integration in the Korobov space with $\alpha=2, b=2, \gamma_j=(0.7)^j$ and $w_j= \left \lfloor{\frac32 \log_b j}\right \rfloor$.}	
	\label{tab:CBC_sequence1}
	\centering
	\begin{tabular}{p{1.8cm}p{1.8cm}p{1.8cm}p{1.8cm}p{1.8cm}p{1.8cm}l}
		\toprule[1.2pt] 
		& $s=50$ & $s=100$ & $s=500$ & $s=1000$ & $s=2000$ & $s^{\ast}$ \\ 
		\toprule[1.2pt] 
		\multirow{2}{4.5em}{$m=10$} & 0.0183 & 0.0329 & 0.163 & 0.32 & 0.64 & \multirow{2}{4em}{101} \\ 
		& \textbf{0.00963} & \textbf{0.00999} & \textbf{0.0106} & \textbf{0.00994} & \textbf{0.0102} \\ 
		\midrule 
		\multirow{2}{4.5em}{$m=12$} 
		& 0.0319 & 0.0485 & 0.239 & 0.475 & 0.944 & \multirow{2}{4em}{255} \\ 
		& \textbf{0.0139} & \textbf{0.0178} & \textbf{0.0295} & \textbf{0.0279} & \textbf{0.0273} \\ 
		\midrule 
		\multirow{2}{4.5em}{$m=14$} 
		& 0.0476 & 0.0899 & 0.425 & 0.861 & 1.74 & \multirow{2}{4em}{645} \\ 
		& \textbf{0.0216} & \textbf{0.0308} & \textbf{0.0806} & \textbf{0.0944} & \textbf{0.0915} \\ 
		\midrule 
		\multirow{2}{4.5em}{$m=16$} 
		& 0.129 & 0.24 & 1.21 & 2.46 & 4.72 & \multirow{2}{4em}{1625} \\ 
		& \textbf{0.0428} & \textbf{0.0744} & \textbf{0.264} & \textbf{0.448} & \textbf{0.619} \\ 
		\midrule 
		\multirow{2}{4.5em}{$m=18$} 
		& 0.424 & 0.829 & 4.14 & 8.45 & 16.8 & \multirow{2}{4em}{4095} \\ 
		& \textbf{0.108} & \textbf{0.178} & \textbf{0.696} & \textbf{1.33} & \textbf{2.65} \\ 
		\midrule 
		\multirow{2}{4.5em}{$m=20$} 
		& 2.23 & 4.22 & 21.5 & 43.2 & 87.2 & \multirow{2}{4em}{10321} \\ 
		& \textbf{0.484} & \textbf{0.839} & \textbf{3.91} & \textbf{7.11} & \textbf{14.2} \\ 
		\midrule
	\end{tabular}
\end{table}

\begin{table}[H]
	\captionof{table}{Computation times (in seconds) for constructing the generating vector $\bsz$ using the unreduced (normal font) and reduced SCS (\textbf{bold font}) construction. The associated lattice can be used for integration in the Korobov space with $\alpha=2, b=2, \gamma_j=(0.7)^j$ and $w_j= \left \lfloor{\frac32 \log_b j}\right \rfloor$.}	
	\label{tab:SCS_sequence1}
	\centering
	\begin{tabular}{p{1.8cm}p{1.8cm}p{1.8cm}p{1.8cm}p{1.8cm}p{1.8cm}l}
		\toprule[1.2pt] 
		& $s=50$ & $s=100$ & $s=500$ & $s=1000$ & $s=2000$ & $s^{\ast}$ \\ 
		\toprule[1.2pt] 
		\multirow{2}{6em}{$m=10$} 
		& 0.0311 & 0.0524 & 0.258 & 0.509 & 1.02 & \multirow{2}{4em}{101} \\ 
		& \textbf{0.0186} & \textbf{0.0155} & \textbf{0.0159} & \textbf{0.016} & \textbf{0.0155} \\ 
		\midrule 
		\multirow{2}{6em}{$m=12$} 
		& 0.0458 & 0.0763 & 0.381 & 0.744 & 1.5 & \multirow{2}{4em}{255} \\ 
		& \textbf{0.0249} & \textbf{0.0316} & \textbf{0.0469} & \textbf{0.0468} & \textbf{0.0458} \\ 
		\midrule 
		\multirow{2}{6em}{$m=14$} 
		& 0.088 & 0.14 & 0.687 & 1.37 & 2.75 & \multirow{2}{4em}{645} \\ 
		& \textbf{0.0427} & \textbf{0.0668} & \textbf{0.175} & \textbf{0.205} & \textbf{0.196} \\ 
		\midrule 
		\multirow{2}{6em}{$m=16$} 
		& 0.202 & 0.397 & 1.93 & 3.88 & 7.73 & \multirow{2}{4em}{1625} \\ 
		& \textbf{0.0838} & \textbf{0.148} & \textbf{0.536} & \textbf{0.89} & \textbf{1.18} \\ 
		\midrule 
		\multirow{2}{6em}{$m=18$} 
		& 0.685 & 1.33 & 6.58 & 13.2 & 27.1 & \multirow{2}{4em}{4095} \\ 
		& \textbf{0.217} & \textbf{0.376} & \textbf{1.54} & \textbf{2.91} & \textbf{5.69} \\ 
		\midrule 
		\multirow{2}{6em}{$m=20$} 
		& 3.33 & 6.62 & 33.5 & 65.9 & 135 & \multirow{2}{4em}{10321} \\ 
		& \textbf{1.06} & \textbf{1.9} & \textbf{8.7} & \textbf{16.7} & \textbf{32.5} \\ 
		\midrule 
	\end{tabular}
\end{table}

\begin{table}[H]
	\captionof{table}{Computation times (in seconds) for constructing the generating vector $\bsz$ using the unreduced (normal font) and reduced CBC (\textbf{bold font}) construction. The associated lattice can be used for integration in the Korobov space with $\alpha=2, b=2, \gamma_j=(0.7)^j$ and $w_j= \left \lfloor{3 \log_b j}\right \rfloor$.}	
	\label{tab:CBC_sequence2}
	\centering
	\begin{tabular}{p{1.8cm}p{1.8cm}p{1.8cm}p{1.8cm}p{1.8cm}p{1.8cm}l}
		\toprule[1.2pt] 
		& $s=50$ & $s=100$ & $s=500$ & $s=1000$ & $s=2000$ & $s^{\ast}$ \\ 
		\toprule[1.2pt] 
		\multirow{2}{6em}{$m=10$} 
		& 0.0173 & 0.0329 & 0.16 & 0.323 & 0.636 & \multirow{2}{4em}{10} \\ 
		& \textbf{0.00298} & \textbf{0.00206} & \textbf{0.00218} & \textbf{0.00222} & \textbf{0.00241} \\ 
		\midrule 
		\multirow{2}{6em}{$m=12$} 
		& 0.0256 & 0.0481 & 0.241 & 0.48 & 0.953 & \multirow{2}{4em}{15} \\ 
		& \textbf{0.00358} & \textbf{0.00365} & \textbf{0.0037} & \textbf{0.00354} & \textbf{0.00439} \\ 
		\midrule 
		\multirow{2}{6em}{$m=14$} 
		& 0.0469 & 0.0851 & 0.438 & 0.856 & 1.88 & \multirow{2}{4em}{25} \\ 
		& \textbf{0.00803} & \textbf{0.00761} & \textbf{0.0105} & \textbf{0.00712} & \textbf{0.00747} \\ 
		\midrule 
		\multirow{2}{6em}{$m=16$} 
		& 0.14 & 0.239 & 1.33 & 2.49 & 5.05 & \multirow{2}{4em}{40} \\ 
		& \textbf{0.0237} & \textbf{0.0233} & \textbf{0.0233} & \textbf{0.0227} & \textbf{0.0251} \\ 
		\midrule 
		\multirow{2}{6em}{$m=18$} 
		& 0.443 & 0.832 & 4.44 & 8.54 & 17.1 & \multirow{2}{4em}{63} \\ 
		& \textbf{0.0798} & \textbf{0.0897} & \textbf{0.0915} & \textbf{0.091} & \textbf{0.09} \\ 
		\midrule 
		\multirow{2}{6em}{$m=20$} 
		& 2.17 & 4.17 & 21.5 & 42.4 & 84.3 & \multirow{2}{4em}{101} \\ 
		& \textbf{0.38} & \textbf{0.623} & \textbf{0.643} & \textbf{0.636} & \textbf{0.628} \\ 
		\midrule
	\end{tabular}
\end{table}

\begin{table}[H]
	\captionof{table}{Computation times (in seconds) for constructing the generating vector $\bsz$ using the unreduced (normal font) and reduced SCS (\textbf{bold font}) construction. The associated lattice can be used for integration in the Korobov space with $\alpha=2, b=2, \gamma_j=(0.7)^j$ and $w_j= \left \lfloor{3 \log_b j}\right \rfloor$.}	
	\label{tab:SCS_sequence2}
	\centering
	\begin{tabular}{p{1.8cm}p{1.8cm}p{1.8cm}p{1.8cm}p{1.8cm}p{1.8cm}l}
		\toprule[1.2pt] 
		& $s=50$ & $s=100$ & $s=500$ & $s=1000$ & $s=2000$ & $s^{\ast}$ \\ 
		\toprule[1.2pt] 
		\multirow{2}{6em}{$m=10$} 
		& 0.0275 & 0.0516 & 0.256 & 0.516 & 1.03 & \multirow{2}{4em}{10} \\ 
		& \textbf{0.00408} & \textbf{0.00327} & \textbf{0.00354} & \textbf{0.00347} & \textbf{0.00329} \\ 
		\midrule 
		\multirow{2}{6em}{$m=12$} 
		& 0.0418 & 0.0751 & 0.383 & 0.756 & 1.56 & \multirow{2}{4em}{15} \\ 
		& \textbf{0.00592} & \textbf{0.00504} & \textbf{0.00612} & \textbf{0.00516} & \textbf{0.00794} \\ 
		\midrule 
		\multirow{2}{6em}{$m=14$} 
		& 0.0792 & 0.14 & 0.767 & 1.39 & 2.82 & \multirow{2}{4em}{25} \\ 
		& \textbf{0.014} & \textbf{0.0136} & \textbf{0.0163} & \textbf{0.0138} & \textbf{0.0138} \\ 
		\midrule 
		\multirow{2}{6em}{$m=16$} 
		& 0.204 & 0.388 & 2.09 & 4.05 & 8.04 & \multirow{2}{4em}{40} \\ 
		& \textbf{0.0441} & \textbf{0.0434} & \textbf{0.0434} & \textbf{0.0423} & \textbf{0.0462} \\ 
		\midrule 
		\multirow{2}{6em}{$m=18$} 
		& 0.686 & 1.35 & 6.89 & 13.7 & 26.8 & \multirow{2}{4em}{63} \\ 
		& \textbf{0.16} & \textbf{0.177} & \textbf{0.182} & \textbf{0.183} & \textbf{0.187} \\ 
		\midrule 
		\multirow{2}{6em}{$m=20$} 
		& 3.28 & 6.71 & 34.4 & 67.4 & 132 & \multirow{2}{4em}{101} \\ 
		& \textbf{0.843} & \textbf{1.4} & \textbf{1.51} & \textbf{1.37} & \textbf{1.36} \\ 
		\midrule
	\end{tabular}
\end{table}

According to Theorem \ref{thm:computational_complexity_lat} and \cite{DKLP15}, the reduced fast CBC and SCS algorithm both construct
a generating vector in $\mathcal{O} \left(m b^m + \min\{s,s^{\ast}\} \, b^m + \sum_{d=1}^{\min\{s,s^{\ast}\}} (m-w_d) b^{m-w_d} \right)$ operations while the unreduced constructions require $\mathcal{O} \left(s m b^m\right)$ operations. Tables 1 to 4 illustrate a drastic reduction of the construction cost between the classic (fast) unreduced CBC and SCS constructions and their reduced counterparts. The magnitude of this speed-up depends on the chosen reduction indices $w_j$. For values of $N=2^{18}$ and $N=2^{20}$ and dimensions $s=1000$ and $s=2000$, the reduction factor ranges from $4$ to $6.5$ and $49$ to $190$ for reductions of $w_j= \left \lfloor{\frac32 \log_b j}\right \rfloor$ and $w_j= \left \lfloor{3 \log_b j}\right \rfloor$, respectively. The higher the dimension $s$ is, the larger the reduction of the computational cost becomes. Furthermore,
the results in Tables 1 to 4 reveal that, for a particular fixed $m$, the computational complexity of the reduced constructions is linear in $s$
as long as $s < s^{\ast}$ and becomes independent of the dimension for $s \ge s^{\ast}$. By considering the cases where $s < s^{\ast}$, this also
demonstrates that a certain part of the achieved reduction originates from reducing the size of the search space for each component $z_j$.

We further note that the speed-up for the reduced fast CBC construction is higher than for the reduced fast SCS algorithm, however, both lie in similar ranges. Our results show that the reduced constructions yield a considerable reduction of the computational cost while the deterioration of the associated error values is only marginal (see Subsections \ref{subsec: error_convergence} and  \ref{subsec: worst-case error}). Similar results have been observed in \cite{DKLP15} for the reduced fast CBC construction. We would like to stress that, as we expected, our implementations of the different CBC constructions in Matlab appear to be much faster than the CBC algorithms used in \cite{DKLP15} (cf. Table \ref{tab:CBC_sequence1}) which were implemented in Mathematica.

\subsection{Analysis of the worst-case errors} \label{subsec: worst-case error}

We investigate the precise values of the worst-case errors $e_{N,s}(\bsz)$ for generating vectors $\bsz$ constructed by the reduced and unreduced
CBC constructions and the reduced SCS construction. Based on the results in Subsection \ref{subsec: error_convergence}, we expect the error values of the two reduced algorithms to be very similar. Let $b=3$ and $N=b^m$ and consider 
a dimension of $s=100$. In Tables \ref{tab:error_values_seq1} and \ref{tab:error_values_seq2} we display the results of numerical tests for different weight sequences 
$\bsgamma = (\gamma_j)_{j \in \NN}$ and reduction indices $w_j= \left \lfloor{\frac32 \log_b j}\right \rfloor$
and $w_j= \left \lfloor{\frac52 \log_b j}\right \rfloor$. For the construction of $\bsz$ via the reduced SCS algorithm we have to choose
suitable initial vectors $\bsz^0$ of the form (\ref{form_initial_vector}). In our experiments we thus consider $q$ different seed vectors with
\begin{align*}
	\bsz^0
	&=
	(Y_1 \bar{z}_1,\ldots,Y_s \bar{z}_s),
\end{align*} 
where the $\bar{z}_j$ are uniform random draws from the set $\mathcal{Z}_{N,w_j}$ for all $j \in \{1,\ldots,s\}$. The reduced fast SCS
algorithm is then applied to all of these $q$ seed vectors $\bsz^0_1,\ldots,\bsz^0_q$ yielding generating vectors 
$\bsz^1_1,\ldots,\bsz^1_q$. The smallest associated worst-case error of these $q$ vectors is then displayed in the tables below. 
Note that the construction cost for this procedure is $\mathcal{O} \left( q m b^m + q \min\{s,s^{\ast}\} \, b^m \right)$, which is feasible for small $q$.
Additionally, we consider the behavior of the reduced SCS algorithm when applied iteratively to the previous outcome 
of the algorithm, i.e., we apply the SCS algorithm to the outcomes $\bsz^1_1,\ldots,\bsz^1_q$ which yields generating vectors 
$\bsz^2_1,\ldots,\bsz^2_q$ that are then again used as seeds for the next iteration and so on, until the algorithm converges to some generating vectors $\bsz_1, \ldots, \bsz_q$. Our empirical numerical experiments suggest that, for the considered cases, this procedure already converges after two runs of the reduced SCS algorithm. Thus, the construction cost for this approach only increases by a factor of $2$.

\begin{table}[H]
	\captionof{table}{$\log_{10}$-worst-case errors $\log_{10} e_{N,s}(\bsz)$ with generating vector $\bsz$ being either constructed via the unreduced (normal font) and reduced ($\overline{\text{overlined}}$) CBC construction or being the best vector constructed by the reduced SCS algorithm with a single run (\textbf{bold font}) or multiple runs ($\underline{\text{underlined}}$). The errors are computed for the Korobov space with $\alpha=2, s=100, b=3, q=100$ and $w_j= \left \lfloor{\frac32 \log_b j}\right \rfloor$.}	
	\centering
	\begin{tabular}{p{2.2cm}p{1.8cm}p{1.8cm}p{1.8cm}p{1.8cm}p{1.8cm}p{1.8cm}} 
		\toprule[1.2pt] 
		& $m = 6$ & $m = 7$ & $m = 8$ & $m = 9$ & $m = 10$ & $m = 11$ \\ 
		\toprule[1.2pt] 
		\multirow{3}{6em}{$\gamma_j=(0.7)^j$} 
		& $-0.4281$ & $-0.7065$ & $-0.9928$ & $-1.283$ & $-1.58$ & $-1.881$ \\ 
		& $\overline{-0.4033}$ & $\overline{-0.685}$ & $\overline{-0.9783}$ & $\overline{-1.265}$ & $\overline{-1.564}$ & $\overline{-1.869}$ \\ 
		& $\mathbf{-0.418}$ & $\mathbf{-0.6934}$ & $\mathbf{-0.9783}$ & $\mathbf{-1.266}$ & $\mathbf{-1.559}$ & $\mathbf{-1.865}$ \\ 
		& $\underline{-0.418}$ & $\underline{-0.6937}$ & $\underline{-0.9783}$ & $\underline{-1.266}$ & $\underline{-1.561}$ & $\underline{-1.865}$ \\
		\midrule 
		\multirow{3}{6em}{$\gamma_j=(0.5)^j$} 
		& $-1.442$ & $-1.804$ & $-2.162$ & $-2.521$ & $-2.889$ & $-3.271$ \\ 
		& $\overline{-1.404}$ & $\overline{-1.771}$ & $\overline{-2.145}$ & $\overline{-2.502}$ & $\overline{-2.879}$ & $\overline{-3.254}$ \\ 
		& $\mathbf{-1.422}$ & $\mathbf{-1.783}$ & $\mathbf{-2.138}$ & $\mathbf{-2.497}$ & $\mathbf{-2.863}$ & $\mathbf{-3.236}$ \\ 
		& $\underline{-1.423}$ & $\underline{-1.783}$ & $\underline{-2.138}$ & $\underline{-2.497}$ & $\underline{-2.864}$ & $\underline{-3.236}$ \\
		\midrule 
		\multirow{3}{6em}{$\gamma_j=1/j^3$} 
		& $-1.754$ & $-2.146$ & $-2.532$ & $-2.923$ & $-3.317$ & $-3.711$ \\ 
		& $\overline{-1.602}$ & $\overline{-2.008}$ & $\overline{-2.452}$ & $\overline{-2.817}$ & $\overline{-3.258}$ & $\overline{-3.66}$ \\ 
		& $\mathbf{-1.618}$ & $\mathbf{-2.037}$ & $\mathbf{-2.441}$ & $\mathbf{-2.851}$ & $\mathbf{-3.245}$ & $\mathbf{-3.635}$ \\ 
		& $\underline{-1.619}$ & $\underline{-2.037}$ & $\underline{-2.441}$ & $\underline{-2.851}$ & $\underline{-3.245}$ & $\underline{-3.637}$ \\ 
		\midrule 
		\multirow{3}{6em}{$\gamma_j=1/j^6$} 
		& $-2.44$ & $-2.904$ & $-3.364$ & $-3.83$ & $-4.286$ & $-4.75$ \\ 
		& $\overline{-2.439}$ & $\overline{-2.904}$ & $\overline{-3.364}$ & $\overline{-3.828}$ & $\overline{-4.288}$ & $\overline{-4.749}$ \\ 
		& $\mathbf{-2.44}$ & $\mathbf{-2.904}$ & $\mathbf{-3.365}$ & $\mathbf{-3.831}$ & $\mathbf{-4.288}$ & $\mathbf{-4.749}$ \\ 
		& $\underline{-2.44}$ & $\underline{-2.904}$ & $\underline{-3.365}$ & $\underline{-3.831}$ & $\underline{-4.288}$ & $\underline{-4.749}$ \\
		\midrule 
	\end{tabular}
	\label{tab:error_values_seq1}
\end{table}

\begin{table}[H]
	\captionof{table}{$\log_{10}$-worst-case errors $\log_{10} e_{N,s}(\bsz)$ with generating vector $\bsz$ being either constructed via the unreduced (normal font) and reduced ($\overline{\text{overlined}}$) CBC construction or being the best vector constructed by the reduced SCS algorithm with a single run (\textbf{bold font}) or multiple runs ($\underline{\text{underlined}}$). The errors are computed for the Korobov space with $\alpha=2, s=100, b=3, q=100$ and $w_j= \left \lfloor{\frac52 \log_b j}\right \rfloor$.}	
	\centering
	\begin{tabular}{p{2.2cm}p{1.8cm}p{1.8cm}p{1.8cm}p{1.8cm}p{1.8cm}p{1.8cm}} 
		\toprule[1.2pt] 
		& $m = 6$ & $m = 7$ & $m = 8$ & $m = 9$ & $m = 10$ & $m = 11$ \\ 
		\toprule[1.2pt] 
		\multirow{3}{6em}{$\gamma_j=(0.7)^j$} 
		& $-0.4281$ & $-0.7065$ & $-0.9928$ & $-1.283$ & $-1.58$ & $-1.881$ \\ 
		& $\overline{-0.1983}$ & $\overline{-0.5021}$ & $\overline{-0.807}$ & $\overline{-1.122}$ & $\overline{-1.426}$ & $\overline{-1.747}$ \\ 
		& $\mathbf{-0.2023}$ & $\mathbf{-0.5136}$ & $\mathbf{-0.8233}$ & $\mathbf{-1.129}$ & $\mathbf{-1.437}$ & $\mathbf{-1.747}$ \\ 
		& $\underline{-0.2029}$ & $\underline{-0.5145}$ & $\underline{-0.8277}$ & $\underline{-1.134}$ & $\underline{-1.442}$ & $\underline{-1.75}$ \\
		\midrule 
		\multirow{3}{6em}{$\gamma_j=(0.5)^j$} 
		& $-1.442$ & $-1.804$ & $-2.162$ & $-2.521$ & $-2.889$ & $-3.271$ \\ 
		& $\overline{-1.113}$ & $\overline{-1.515}$ & $\overline{-1.901}$ & $\overline{-2.33}$ & $\overline{-2.703}$ & $\overline{-3.11}$ \\ 
		& $\mathbf{-1.116}$ & $\mathbf{-1.524}$ & $\mathbf{-1.931}$ & $\mathbf{-2.317}$ & $\mathbf{-2.709}$ & $\mathbf{-3.094}$ \\ 
		& $\underline{-1.119}$ & $\underline{-1.527}$ & $\underline{-1.931}$ & $\underline{-2.325}$ & $\underline{-2.717}$ & $\underline{-3.1}$ \\ 
		\midrule 
		\multirow{3}{6em}{$\gamma_j=1/j^3$} 
		& $-1.754$ & $-2.146$ & $-2.532$ & $-2.923$ & $-3.317$ & $-3.711$ \\ 
		& $\overline{-0.9724}$ & $\overline{-1.181}$ & $\overline{-1.391}$ & $\overline{-1.622}$ & $\overline{-1.919}$ & $\overline{-2.396}$ \\ 
		& $\mathbf{-0.973}$ & $\mathbf{-1.182}$ & $\mathbf{-1.392}$ & $\mathbf{-1.622}$ & $\mathbf{-1.92}$ & $\mathbf{-2.396}$ \\ 
		& $\underline{-0.973}$ & $\underline{-1.182}$ & $\underline{-1.392}$ & $\underline{-1.622}$ & $\underline{-1.92}$ & $\underline{-2.396}$ \\ 
		\midrule 
		\multirow{3}{6em}{$\gamma_j=1/j^6$} 
		& $-2.44$ & $-2.904$ & $-3.364$ & $-3.83$ & $-4.286$ & $-4.75$ \\ 
		& $\overline{-2.361}$ & $\overline{-2.81}$ & $\overline{-3.268}$ & $\overline{-3.728}$ & $\overline{-4.191}$ & $\overline{-4.657}$ \\ 
		& $\mathbf{-2.362}$ & $\mathbf{-2.811}$ & $\mathbf{-3.269}$ & $\mathbf{-3.728}$ & $\mathbf{-4.191}$ & $\mathbf{-4.657}$ \\ 
		& $\underline{-2.362}$ & $\underline{-2.811}$ & $\underline{-3.269}$ & $\underline{-3.728}$ & $\underline{-4.191}$ & $\underline{-4.657}$ \\ 
		\midrule 
	\end{tabular}
	\label{tab:error_values_seq2}
\end{table}

The results in Tables \ref{tab:error_values_seq1} and \ref{tab:error_values_seq2} show that the reduced fast SCS algorithm 
generates lattice rules with similar errors as the reduced CBC algorithm as was to be expected from the results in Section 
\ref{subsec: error_convergence}. Furthermore, we note that it is possible to obtain better error values than with the reduced CBC construction,
at the price of increased computational costs. It becomes also evident that, in certain cases, applying the reduced SCS algorithm
repeatedly yields even further, though rather small, improvement. However, whether that strategy is successful or not depends strongly on the various involved parameters 
such that a general quantitative statement can currently not be inferred. Theorem \ref{thm: z0-z} ensures that the associated worst-case error never 
increases by repeated runs of the SCS algorithm. A systematic analysis of the effect of repetition of the SCS algorithm 
is left open for future research. The loss of accuracy as compared to the classic CBC construction is for both reduced algorithms only marginal. 
The only exception to this is the case $\gamma_j=1/j^3$ in Table \ref{tab:error_values_seq2}. As discussed in Subsection \ref{subsec: error_convergence}, 
this is most likely due to the fact that the weights $\gamma_j$ do not decay fast enough (cf. Case (c) in Figure 2).

\section{Walsh spaces and polynomial lattice point sets} \label{secWal}

\subsection{Walsh spaces}

Similar results to those for lattice point sets from the previous sections can be shown for polynomial lattice 
point sets over finite fields $\FF_b$ of prime order $b$ with modulus $x^m$. Here we only sketch these results and the necessary notation, 
as they are in analogy to those for Korobov spaces and lattice point sets. 

\medskip

As a quality criterion we use the worst-case error of QMC rules in a weighted Walsh space, as introduced in 
\cite{DP05} for the case of product weights, with general weights. 
For a prime number $b$ and $h \in \NN$ define $\psi_b(h) := \lfloor \log_b(h)\rfloor$. We furthermore write
\begin{align*}
r_\alpha (h)=\begin{cases}
           1 & \mbox{if $h=0$,}\\
           b^{-\alpha \psi_b (h)} & \mbox{if $h\neq 0$,}
          \end{cases}
\end{align*}
for $h\in\NN_0$ and set
\begin{align*}
\mu_b (\alpha):=\sum_{h=1}^\infty r_{\alpha} (h)= \sum_{a=0}^\infty \frac{1}{b^{a\alpha}} \sum_{k=b^a}^{b^{a+1}-1} 1=\sum_{a=0}^\infty \frac{(b-1)b^a}
{b^{a\alpha}}
=\frac{b^{\alpha}(b-1)}{b^\alpha - b}.
\end{align*} 
For the multivariate case with dimension $s\in\NN$ and $\bsh=(h_1,\ldots,h_s)$ we set $r_\alpha (\bsh)=\prod_{j=1}^s r_\alpha (h_j)$. 
Moreover, for a nonnegative integer $h$, we define the $h$-th Walsh function $\ _b\wal_h :[0,1)\To\CC$ by
\begin{align*}
\ _b\wal_h (x):=\mathrm{e}^{2\pi\icomp (x_1 h_0 + \cdots + x_{a+1} h_a)/b}
\end{align*}
with $x\in [0,1)$, and base $b$ representations $h=h_0 + h_1 b + \cdots + h_a b^a$, with $h_i\in\{0,1,\ldots,b-1\}$, and 
$x=\frac{x_1}{b} + \frac{x_2}{b^2}+\cdots$ (unique in the sense that infinitely many of the $x_i$ must be different from $b-1$).

For dimension $s\ge 2$ and vectors $\bsh=(h_1,\ldots,h_s)\in\NN_0^s$, and $\bsx=(x_1,\ldots,x_s)\in [0,1)^s$ we define 
$\ _b\wal_{\bsh}:[0,1)^s\To\CC$ by 
\begin{align*}
\ _b\wal_{\bsh} (\bsx):=\prod_{j=1}^s \ _b\wal_{h_j} (x_j).
\end{align*}
In the following, we will consider the prime number $b$ as fixed, and then simply write $\wal_h$ or $\wal_{\bsh}$ instead of 
$\  _b\wal_h$ or $\ _b\wal_{\bsh}$, respectively. \\

The weighted Walsh space $\cH(K_{s,\alpha,\bsgamma}^{\wal})$ is a reproducing kernel Hilbert space with kernel function of the form 
\begin{align*}
K_{s,\alpha,\bsgamma}^{\wal}(\bsx,\bsy) = 1+ \sum_{\emptyset \not=\uu \subseteq [s]} \gamma_{\uu} \sum_{\bsh_{\uu}\in \NN^{|\uu|}} 
r_\alpha (\bsh_{\uu})\wal_{\bsh_{\uu}}(\bsx_{\uu})\overline{\wal_{\bsh_{\uu}}(\bsy_{\uu})},
\end{align*}
and inner product
\begin{align*}
\langle f,g\rangle_{K_{s,\alpha,\bsgamma}^{\wal}}=\sum_{\uu \subseteq [s]} 
\gamma_{\uu}^{-1} \sum_{\bsh_{\uu}\in \NN^{|\uu|}} \left(r_\alpha (\bsh_{\uu})\right)^{-1} \widetilde{f}((\bsh_{\uu},\bszero)) 
\overline{\widetilde{g}((\bsh_{\uu},\bszero))},
\end{align*}
where $\widetilde{f}(\bsh)=\int_{[0,1]^s} f(\bst) \overline{\wal_{\bsh}(\bst)}\rd \bst$ is the $\bsh$-th Walsh coefficient of $f$
and $(\bsh_{\uu},\bszero) \in \NN^s$ denotes the vector whose $j$-th component is equal to the corresponding one of $\bsh_{\uu}$ if $j \in \uu$ and zero if $j \not\in \uu$.

For integration in $\cH(K_{s,\alpha,\bsgamma}^{\wal})$ we use a special instance of polynomial lattice point sets over the finite field
$\FF_b$ with prime $b$. Polynomial lattice point sets are special examples of $(t,m,s)$-nets in base $b$, which were proposed by Niederreiter in \cite{N92a} (see also \cite[Ch.~4.4]{N92b}). Let $\FF_b ((x^{-1}))$ be the field of formal Laurent series over $\FF_b$ with 
elements of the form
\begin{align*}
L=\sum_{\ell=w}^\infty t_{\ell} x^{-\ell},
\end{align*}
where $w$ is an arbitrary integer and all $t_{\ell}\in\FF_b$. Note that the field of rational functions is a subfield of $\FF_b ((x^{-1}))$. We further 
denote by $\FF_b [x]$ the set of all polynomials over $\FF_b$ and define the map $\nu: \FF_b ((x^{-1}))\To [0,1)$ by
\begin{align*}
\nu \left(\sum_{\ell =w}^\infty t_\ell x^{-\ell}\right)=\sum_{\ell=\max (1,w)}^m t_{\ell} b^{-\ell}.
\end{align*}
There is a close connection between the base $b$ expansions of natural numbers and the polynomial ring $\FF_b [x]$. For $n\in \NN_0$ with $b$-adic expansion $n=n_0 + n_1 b + \cdots + n_a b^{a}$, we associate $n$ with the polynomial
\begin{align*}
n(x):=\sum_{k=0}^a n_k x^k \in \FF_b [x].
\end{align*} 
Now, for given integers $m\ge 1$ and $s\ge 2$, choose $f\in \FF_b [x]$ with $\deg (f)=m$, and let $g_1,\ldots,g_s \in\FF_b [x]$. 
Then the point set $\cP (\bsg,f)$ is defined as the collection of the $b^m$ points 
\begin{align*}
\bsx_n:=\left(\nu\left(\frac{n\ g_1}{f}\right),\ldots,\nu\left(\frac{n\ g_s}{f}\right)\right)\ \ \ \mbox{ for }\ n \in \FF_b[x]\ \mbox{ with } 
\deg(n)<m.
\end{align*}
Note that one can restrict the choice of $g_j$ for $j=1,\ldots,s$ to the set
\begin{align*}
\{ g\in\FF_b [x]:\, \deg(g) < m \}.
\end{align*}
Due to the construction principle, $\cP (\bsg,f)$ is often called a polynomial lattice and a QMC rule using the point
set $\cP (\bsg,f)$ is referred to as a polynomial lattice rule (modulo $f$). The vector $\bsg=(g_1,\ldots,g_s)$ 
is called the generating vector. \\  

For our purposes, we only consider a special case of lattice rules, namely the special choice of $f(x)=x^m$ as
the modulus. With a slight misuse of notation, we shall often write $x^m$ instead of $f$. 
Let now $\cP(\bsg,x^m)$, where $\bsg=(g_1,\ldots,g_s) \in (\FF_b[x])^s$, be the $b^m$-element polynomial lattice consisting of 
\begin{align*}
\bsx_n:=\left(\nu\left(\frac{n\ g_1}{x^m}\right),\ldots,\nu\left(\frac{n\ g_s}{x^m}\right)\right)\ \ \ \mbox{ for }\ n \in \FF_b[x]\ \mbox{ with } 
\deg(n)<m,
\end{align*}
where for $v \in \FF_b[x]$, $v(x)=a_0+a_1 x+\cdots +a_r x^r$, with $\deg(v)=r$, the map $\nu$ is in this particular case given by 
\begin{align*}
\nu\left(\frac{v}{x^m}\right):= \frac{a_{\min(r,m-1)}}{b^{m-\min(r,m-1)}}+\cdots+\frac{a_1}{b^{m-1}}+\frac{a_0}{b^m} \in [0,1).
\end{align*}
Note that $\nu(v/x^m)=\nu((v\pmod{x^m})/x^m)$. We refer to \cite[Chapter~10]{DP10} for more information on 
polynomial lattice point sets. \\

In the following we write, for a nonnegative integer $h$ with base $b$ representation $\sum_{i=0}^a h_i b^i$, 
\begin{align*}
	\tr_m (h) = \tr_m (h)(x) := h_0 + h_1 x + \cdots + h_{m-1} x^{m-1}, 
\end{align*}
where the $h_i$ with $i>a$ are set equal to zero. For vectors of nonnegative integers $\bsh \in \NN^s$, $\tr_m(\bsh) \in (\FF_b[x])^s$ is defined component-wise. \\

The worst-case error of a polynomial lattice rule based on $\cP(\bsg,x^m)$ with $\bsg \in (\FF_b[x])^s$ in the weighted Walsh space $\cH(K_{s,\alpha,
\bsgamma}^{\wal})$ is given by (see \cite{DKPS05})
\begin{equation*}
e_{N,s}^2 (\bsg)=\sum_{\emptyset\neq\uu\subseteq [s]}\gamma_\uu 
\sum_{\bsh_{\uu} \in \D_\uu} \prod_{j \in \uu} b^{-\alpha \psi_b(h_j)},
\end{equation*}
where 
\begin{align*}
	\D_\setu (\bsg_\setu)=\D_\uu:=\left\{\bsh_\uu\in \NN^{\abs{\uu}} :\  \tr_m(\bsh_\uu) \cdot\bsg_\uu\equiv 0\, (x^m)\right\},
\end{align*}
and for $\bsv=(v_1,\ldots,v_s)$ and $\bsu=(u_1,\ldots,u_s)$ in $(\FF_b [x])^s$ we define the vector product by
$\bsv\cdot\bsu:=\sum_{i=1}^s v_i u_i$.

\subsection{The reduced SCS algorithm for polynomial lattice rules}

Let us now assume again that $f(x)=x^m$ for some integer $m$ and that we are given weights $\gamma_\setu$, $\setu\subseteq [s]$, 
and a non-decreasing sequence of integers $w_j$ with $w_1 \le w_2 \le w_3 \le \cdots$. 

Then we define the restricted search set for the $j$-th component of the generating vector $\bsg$ as 	
\begin{align*}
\mathcal{G}_{N,w_j} &= \begin{cases} \{g\in\FF_b [x]:\, 0\le \deg (g) < m-w_j\ \mbox{and}\ \gcd (g,f)=1\} & \mbox{if }  w_j < m,  \\ 
\{1\in\FF_b [x]\} & \mbox{if } w_j \ge m, \end{cases} 
\end{align*}
and note that these sets depend on the integers $w_j$. Additionally, we note that the cardinality of the search space is 
$|\mathcal{G}_{N,w_j}| = b^{m-w_j-1}$.
Moreover, we put $Y_j (x)=x^{w_j}$, and again with a misuse of notation we sometimes write $Y_j =x^{w_j}$. \\

We then consider the following algorithm for the construction of the generating vector $\bsg$ based on some initial vector $\bsg^{0}\in 
(\FF_b[x])^s$. 

\begin{algorithm}\label{alg:polylattice} 

Let $f\in\FF_b [x]$, $f(x)=x^m$ for a fixed $m\in\NN$, let $\gamma_{\setu}$, $\setu\subseteq [s]$ be 
general weights, and let the worst-case error $e_{N,s}$ in the weighted Walsh space 
$\calH (K_{s,\alpha,\bsgamma}^{\wal})$ be defined as above. Furthermore, let $w_1 \le w_2 \le \cdots \le w_s$ and $Y_j(x) = x^{w_j}$ 
for $j\in\{1,\ldots,s\}$. Then we construct the generating vector $\bsg = (Y_1 g_1, \ldots, Y_s g_s)$ as follows.
\begin{itemize}
	\item \textbf{Input:} Starting vector $\bsg^0=(g_1^0,\ldots,g_s^0) \in \{g \in \FF_b [x] :\, \deg (g) < m\}^s$. 
	\item For $d \in [s]$ assume $g_1,\ldots,g_{d-1}$ have already been selected. Then choose $g_d \in \mathcal{G}_{N,w_d}$
	such that $e^2_{N,s}((Y_1 g_1, \ldots, Y_{d-1} g_{d-1}, Y_d g_d, g_{d+1}^0, \ldots, g_s^0))$ 
	is minimized as a function of 
	$g_d$.
	\item Increase $d$ until $g_1,\ldots,g_s$ are found.
\end{itemize}
\end{algorithm}

\begin{theorem} \label{thm: reduced SCS poly}
	Let the assumptions in Algorithm \ref{alg:polylattice} hold.
	Let $\bsg = (Y_1 g_1, \ldots, Y_s g_s)$ be constructed by Algorithm \ref{alg:polylattice}. Then, for $\lambda \in (\frac{1}
	{\alpha},1]$, the squared worst-case error $e^2_{N,s}(\bsg)$ satisfies
	\begin{align*}
		e^2_{N,s}((Y_1 g_1, \ldots, Y_s g_s))
		&\le
		\left( \sum_{d=1}^{s} \sum_{d \in \setu \subseteq [s]} \gamma_{\setu}^{\, \lambda} 
		\frac{2 ( \mu_b (\alpha \lambda))^{|\setu|}}{b^{\max(0,m-w_d)}} \right)^{\frac{1}{\lambda}} .
	\end{align*}
\end{theorem}

\begin{proof}
The proof works analogously to the proof of Theorem \ref{thm: reduced SCS}.
\end{proof}

The following corollary is derived in a similar way from Theorem \ref{thm: reduced SCS poly} 
as Corollary \ref{corlatticegenweights} is derived from Theorem \ref{thm: reduced SCS}. 
\begin{corollary}\label{corpolygenweights}
Let the assumptions in Algorithm \ref{alg:polylattice} hold.
 Let $\bsg = (Y_1 g_1,\ldots,Y_s g_s)$ be constructed by Algorithm \ref{alg:polylattice}. 
Then we have for all $\delta \in (0,\frac{\alpha-1}{2}]$ that
\begin{align*}
	e_{N,s}(\bsg)
	\le
	C_{s,\alpha,\bsgamma,\delta} \; N^{-\alpha/2+\delta},
\end{align*}
where
\begin{align*}
	C_{s,\alpha,\bsgamma,\delta}
	&:=
	\left(2\sum_{d=1}^{s} \sum_{d \in \setu\subseteq [s]} \gamma_{\setu}^{\frac{1}{\alpha-2\delta}} 
	\left(\mu_b \left(\frac{\alpha}{\alpha-2\delta}\right)\right)^{|\setu|} b^{w_d}\right)^{\alpha/2-\delta}.
\end{align*}
    For $\delta\in (0,\frac{\alpha-1}{2}]$ and $q\ge 0$, define 
\begin{align*}
C_{\delta,q}:=\sup_{s\in\NN} \left[\frac{2}{s^q}\sum_{d=1}^{s}\sum_{d \in \setu\subseteq [s]} 
	\gamma_{\setu}^{\frac{1}{\alpha-2\delta}} \left(\mu_b \left(\frac{\alpha}{\alpha-2\delta}\right)\right)^{|\setu|} b^{w_d} \right].
\end{align*}
If $C_{\delta,q}<\infty$ for some $\delta\in (0,\frac{\alpha-1}{2}]$ and $q\ge 0$ then
\begin{align*}
     e_{N,s}(\bsg)\le (s^q C_{\delta,q})^{\alpha/2-\delta} N^{-\alpha/2+\delta}.
\end{align*}
If $C_{\delta,0}<\infty$ for some $\delta\in (0,\frac{\alpha-1}{2}]$ then
\begin{align*}
     e_{N,s}(\bsg)\le (C_{\delta,0})^{\alpha/2-\delta} N^{-\alpha/2+\delta}.
\end{align*}
\end{corollary}

For the following corollary, which is shown analogously to Corollary \ref{corlatticeproduct},
we again assume product weights, i.e., $\gamma_{\setu}=\prod_{j\in\setu} \gamma_j$ for $\setu\subseteq [s]$, 
where the $\gamma_j$ are elements of an infinite, non-increasing sequence of positive reals, $(\gamma_j)_{j\ge 1}$.

\begin{corollary}\label{corpolyproduct}
	Let the assumptions in Algorithm \ref{alg:polylattice} hold.
	Let $\bsg = (Y_1 g_1,\ldots,Y_s g_s)$ be constructed by Algorithm \ref{alg:polylattice}. Then we have for all $\delta \in (0,\frac{\alpha-1}{2}]$
	that
	\begin{align*}
		e_{N,s}(\bsg)
		\le
		C_{s,\alpha,\bsgamma,\delta} \; N^{-\alpha/2 + \delta},
	\end{align*}
	where
	\begin{align*}
		C_{s,\alpha,\bsgamma,\delta}
		&:=
		\left( \left( \sum_{d=1}^{s} \gamma_d^{\frac{1}{\alpha - 2 \delta}} b^{w_d} \right) \left( 2 \mu_b\left(\frac{\alpha}
		{\alpha - 2 \delta}\right) \right) \prod_{d=1}^{s-1} \left( 1 +  \gamma_d^{\frac{1}{\alpha - 2 \delta}}   
		\mu_b\left(\frac{\alpha}{\alpha - 2 \delta}\right) \right) \right)^{\alpha/2 - \delta} .
	\end{align*}
	Furthermore, the constant $C_{s,\alpha,\bsgamma,\delta}$ is bounded independently of the dimension $s$ if
	\begin{align*}
	\sum_{d=1}^{\infty} \gamma_d^{\frac{1}{\alpha - 2 \delta}} b^{w_d} < \infty .
	\end{align*}
\end{corollary}

By setting all $w_j$ equal to zero in Algorithm \ref{alg:polylattice}, we obtain an (unreduced) SCS algorithm, and the corresponding 
analogous results in Theorem \ref{thm: reduced SCS poly} and Corollaries \ref{corpolygenweights} and \ref{corpolyproduct}. We would like to point out that an SCS algorithm for the 
polynomial lattice case has not existed previously.

\begin{theorem} \label{thm:SCSgeneralpoly}
    Let $f\in\FF_b[x]$, $f(x)=x^m$ for a fixed $m\in\NN$, let $\gamma_{\setu}$, $\setu\subseteq [s]$, be general weights, and let the worst-case error 
    $e_{N,s}$ in the weighted Walsh space $\calH (K_{s,\alpha,\bsgamma})$ be defined as above. Let $\bsg^{0}\in(\FF_b [x])^s$ be an arbitrary initial vector. Then Algorithm \ref{alg:polylattice} applied with $w_1=\cdots=w_s=0$ constructs $\bsg = (g_1, \ldots, g_s)$ such that, for $\lambda \in (\frac{1}{\alpha},1]$, the squared worst-case error $e^2_{N,s}(\bsg)$ satisfies
	\begin{align*}
		e^2_{N,s}((g_1, \ldots, g_s))
		&\le
		\left( \sum_{\emptyset \ne \setu \subseteq [s]} \abs{\setu}
		\gamma_{\setu}^{\, \lambda} \frac{2 (\mu_b(\alpha \lambda))^{|\setu|}}{b^{m}} \right)^{\frac{1}{\lambda}} 
		.
	\end{align*}
	In particular, $e_{N,s} (\bsg)\in \mathcal{O} (N^{-\alpha/2+\delta})$ for $\delta$ arbitrarily close to zero, where the implied constant is independent 
	of $s$ if
	\[
	C_{\delta}
	:=
	\sup_{s\in\NN} \left[2 \sum_{\emptyset \ne \setu\subseteq [s]} \abs{\setu} 
	\gamma_{\setu}^{\frac{1}{\alpha-2\delta}} \left(\mu_b\left(\frac{\alpha}{\alpha-2\delta}\right)\right)^{|\setu|}  \right]
	<
	\infty
	.
	\]
\end{theorem}

\subsection{Fast implementation of the reduced SCS algorithm for polynomial lattice points}

By using the same theory that was used in \cite[Section 5]{DKLP15}, it is possible to obtain a fast implementation of the SCS algorithm also for the polynomial lattice rule case. 
Indeed, the precomputation outlined in Algorithm \ref{alg:fast_red_scs} can be done similarly for polynomial lattice points by using an analogous error expression that 
was shown in \cite{DKPS05}. Furthermore, as outlined for the reduced CBC construction of polynomial lattice rules in \cite{DKLP15},
the matrix-vector multiplication can be implemented such that it uses a number of operations that exceeds the order of magnitude in the lattice case only by one logarithmic factor. These observations lead to the following theorem.

\begin{theorem}
Algorithm \ref{alg:polylattice} can be implemented such that its computational cost is of order 
\begin{align*}
	\mathcal{O} \left( m b^m + \min\{s,s^{\ast}\} \, b^m + \sum_{j=1}^{\min\{s,s^{\ast}\}} (m-w_d)^2 b^{m-w_d} \right).
\end{align*}
\end{theorem}

\section{Conclusion}

In this paper, we studied a combination of the SCS algorithm introduced in \cite{ELN18}, and the reduced construction approach introduced in \cite{DKLP15}, with the goal of pooling the advantages of these two methods: by the reduced construction method, we can drastically reduce the computational cost compared to the unreduced algorithm, and by an SCS construction we may obtain better numerical error values for the corresponding integration rules. We showed that our new algorithm yields generating vectors of lattice rules achieving an almost optimal convergence rate, where the weights in the function space can help in overcoming the curse of dimensionality. By our new results, we extended previous results to arbitrary weights and non-prime numbers of points. 
Furthermore, the considered algorithms were implemented in an efficient way using a modern programming language; numerical tests confirm our main results. Similar observations hold for the case of polynomial lattice rules. It would be interesting to study further improvements on CBC or SCS algorithms, for example the choice of good initial vectors $\bsz^0$. In future research, we will consider the use of reduced CBC and SCS construction methods for the special choice of product and order dependent (POD) weights, i.e., 
weights of the form $\gamma_{\setu} = \Gamma_{|\setu|} \prod_{j \in \setu} \gamma_j$, where the $\Gamma_{|\setu|}$ only depend 
on the cardinality of $\setu$ (see, e.g., \cite{KSS12}), and the application of the obtained lattice rules in PDE problems.

\section*{Acknowledgements}
The authors would like to thank two anonymous referees for helpful comments regarding an improved presentation of the results. Moreover, the authors are grateful for
Friedrich Pillichshammer's remarks which enhanced the exposition of the obtained results.

P.~Kritzer is supported by the Austrian Science Fund (FWF): Project F5506-N26, which is part of the Special Research Program ``Quasi-Monte Carlo Methods: Theory and Applications''. P. Kritzer furthermore gratefully acknowledges the partial support of the Erwin Schr\"odinger International Institute for Mathematics and Physics (ESI) in Vienna under the thematic programme ``Tractability of High Dimensional Problems and Discrepancy''.

The authors acknowledge the support of the National Science Foundation (NSF) under Grant DMS-1638521 to the	Statistical and Applied Mathematical Sciences Institute.

\begin{small}
\noindent\textbf{Authors' addresses:}\\

 \noindent Adrian Ebert\\
 Department of Computer Science\\
 KU Leuven\\
 Celestijnenlaan 200A, 3001 Leuven, Belgium.\\
 \texttt{adrian.ebert@cs.kuleuven.be}

 \medskip
 
 \noindent Peter Kritzer\\
 Johann Radon Institute for Computational and Applied Mathematics (RICAM)\\
 Austrian Academy of Sciences\\
 Altenbergerstr. 69, 4040 Linz, Austria.\\
 \texttt{peter.kritzer@oeaw.ac.at}

 \end{small}

\end{document}